\documentclass[12pt]{amsart}
\usepackage{amsmath, amscd}
\usepackage{amsthm}
\usepackage{verbatim}
\usepackage{amssymb}
\usepackage{hyperref}
\usepackage[all, cmtip]{xy}

\newtheorem{lemma}{Lemma}[section]
\newtheorem{theorem}[lemma]{Theorem}
\newtheorem{corollary}[lemma]{Corollary}
\newtheorem{proposition}[lemma]{Proposition}

\theoremstyle{definition}
\newtheorem{definition}[lemma]{Definition}
\newtheorem{remark}[lemma]{Remark}

\newtheorem{example}[lemma]{Example}

\theoremstyle{remark}
\newtheorem*{remark*}{Remark}
\newtheorem*{note*}{Note}

\newcommand{\Cone}{\operatorname{Cone}}

\newcommand{\Bl}{\operatorname{Bl}}
\newcommand{\GL}{\operatorname{GL}}

\newcommand{\Spec}{\operatorname{Spec}}
\newcommand{\Proj}{\operatorname{Proj}}
\newcommand{\Hom}{\operatorname{Hom}}

\newcommand{\ns}{\operatorname{ns}}

\newcommand{\DG}{\operatorname{DG}}
\newcommand{\mc}[1]{\mathcal{#1}}
\newcommand{\s}{\sigma}

\newcommand{\CC}{\mathbb{C}}
\newcommand{\Af}{\mathbb{A}}
\newcommand{\VV}{\mathbb{V}}
\newcommand{\ZZ}{\mathbb{Z}}
\newcommand{\RR}{\mathbb{R}}
\newcommand{\QQ}{\mathbb{Q}}

\newcommand{\A}{\mathbb{A}}

\newcommand{\Pro}{\mathbb{P}}

\newcommand{\CCe}{{\CC^*}}
\newcounter{item-counter}
\newcommand{\map}[3]{\xymatrix{#2 \ar[r]^{#1} &#3}}
\newcommand{\mbf}[1]{{\mathbf{#1}}} 

\begin{document}
\title[Partial desingularizations of good moduli spaces]{Partial
  desingularizations of good moduli spaces of Artin toric
  stacks}

\author{Dan Edidin}
\address{Department of Mathematics, University of Missouri-Columbia, Columbia, Missouri 65211}
\email{edidind@missouri.edu}

\author{Yogesh More}
\address{Department of Mathematics, University of Missouri-Columbia, Columbia, Missouri 65211}
\email{yogeshmore80@gmail.com}
\thanks{The first author was partially supported by NSA grant 
H98230-08-1-0059 while preparing this article.}

\begin{abstract}
  Let ${\mathcal X}$ be an Artin stack with good moduli space
  ${\mathcal X}\to M$. We define the Reichstein transform of
  ${\mathcal X}$ relative to a closed substack ${\mathcal C} \subset
  {\mathcal X}$ to be the complement of the strict transform of the
  saturation of ${\mathcal C}$ in the blowup of ${\mathcal X}$ along
  ${\mathcal C}$.  The main technical result of the paper is that the
  Reichstein transform of a toric Artin stack relative to a toric
  substack is again a toric stack. Precisely, the Reichstein transform
  relative to a cone in a stacky fan is the toric stack determined by
  {\em stacky star subdivision}. This leads to our main theorem which
  states that for toric Artin stacks there is a canonical sequence of
  Reichstein transforms that produces a toric Deligne-Mumford
  stack. When the good moduli space of the toric stack is a projective
  toric variety our procedure can be interpreted in terms of Kirwan's
  \cite{Kir:85} partial desingularization of geometric invariant
  theory quotients.
\end{abstract}

\maketitle

\section{Introduction}
If $G$ is a linear algebraic group acting properly on a smooth
algebraic variety $X$ then there is a geometric quotient $X \to X/G$,
where $X/G$ is an algebraic space with only finite quotient
singularities (cf. \cite{KeMo:97,Kol:97}). Unfortunately, if $G$ does
not act properly there is no general method for determining if a
quotient exists, even as an algebraic space. However, if $G$ is
reductive and $X$ is a projective variety then geometric invariant
theory can be used to construct quotients of the open sets $X^{s}
\subset X^{ss}$ of the $G$-stable and semi-stable points. The quotient
$X^{ss}/G$ is projective, but in general highly singular, while the
open subspace $X^{s}/G$ has only finite quotient singularities.

In a landmark paper \cite{Kir:85} Kirwan described, when $X^{s}\neq
\emptyset$,  a systematic sequence of blowups along non-singular $G$-invariant subvarieties that yields a birational $G$-equivariant morphism $f:X' \to X$ such that every semi-stable point (with respect to a suitable linearization) of the $G$-variety $X'$ is stable. 
The quotient $(X')^{ss}/G$ is a
projective variety with only finite
quotient singularities. Furthermore, there is an induced projective birational morphism $\overline{f}:(X')^{ss}/G \to X^{ss}/G$, which is an isomorphism over the open
set $X^{s}/G$. Hence the quotient $(X')^{ss}/G$ may be viewed
as a partial resolution of singularities of the highly singular
quotient $X^{ss}/G$.

A natural problem is to try to understand to what extent Kirwan's
procedure can be replicated for non-projective quotients where the
techniques of geometric invariant theory do not apply. Precisely,
given a smooth $G$-variety $X$ and a good quotient (see Section
\ref{subsec.goodquotient} for definitions) $X \stackrel{q} \to X/G$ we
would like to find a systematic way of producing a birational map $X'
\to X$ such that $G$ acts properly on $X'$ and the induced map of
quotients $X'/G \to X/G$ is proper. More generally suppose ${\mathcal X}$
is a smooth
algebraic stack with a good moduli space (in the sense
of \cite{Alp:08}) ${\mathcal X} \to M$. Is there a systematic way of
producing a smooth separated Deligne-Mumford stack ${\mathcal X}'$ and a
morphism ${\mathcal X'} \to {\mathcal X}$ which is generically an
isomorphism, such that the induced map $M' \to M$ is proper and
birational, where $M'$ is the coarse moduli space of the
Deligne-Mumford stack ${\mathcal X'}$?

The main result of this paper (Theorem \ref{thm.torickirwan})  is to solve this
problem when $\mc{X}$ is an Artin toric stack (as defined by Borisov, Chen and Smith in
\cite{BCS:05}). From our perspective toric stacks are a class of
stacks with good moduli spaces which are not in general geometric
invariant theory quotients.

To prove our result we introduce certain birational transformations of
Artin stacks with good moduli spaces we call Reichstein transforms
(Section \ref{subsec.reichstein}). If ${\mathcal C} \subset {\mathcal X}$
is a closed substack then the {\em Reichstein transform} ${\mathcal
  R}({\mathcal X}, {\mathcal C})$ of ${\mathcal X}$ relative to
${\mathcal C}$ is defined to be the complement of the strict transform
of the saturation of ${\mathcal C}$ (relative to the quotient map
$q\colon {\mathcal X} \to M$) in the blowup of ${\mathcal X}$ along
${\mathcal C}$.  The main technical result of the paper is Theorem
\ref{thm.technical}. It states that the Reichstein transform of a
toric stack along a toric substack produces a toric stack
combinatorially related to the original stack by {\em stacky star
  subdivision}.  Successively applying Theorem \ref{thm.technical}
yields the procedure for obtaining a toric Deligne-Mumford stack.

Finally, in Section \ref{sec.divreich} we discuss the relationship
between divisorial Reichstein transformations and changes of
linearizations in geometric invariant theory. We prove (Theorem
\ref{thm.divsub}) that every toric stack contains a separated
Deligne-Mumford substack which can be obtained via a (non-canonical)
sequence of Reichstein transformations relative to divisors. We also
give an example (Example \ref{ex.secretchamber}) of a toric Artin
stack which has a projective good moduli space (and so is in fact a
geometric invariant theory quotient) but which contains a complete open
Deligne-Mumford toric stack whose moduli space is a complete
non-projective toric variety.

In a subsequent paper we will study the behavior of Reichstein transforms
on arbitrary quotient stacks with good moduli spaces.

{\bf Acknowledgments:} The authors are grateful to Jarod Alper and Johan de Jong for helpful discussions while preparing this article.

\section{Good quotients and Reichstein transformations}

\subsection{Standing assumptions}
Throughout this paper we work over the field $\CC$ of complex numbers.
All algebraic groups are assumed to be linear, that
is, isomorphic to subgroups of $\GL_n(\CC)$ for some $n$. 
Because we work in characteristic 0, all reductive groups are linearly reductive - that is every representation decomposes into a direct sum of irreducibles.

\subsection{Good quotients and good moduli spaces} \label{subsec.goodquotient}

Let $G$ be an
algebraic group acting on an algebraic variety $X$ (or more generally
an algebraic space). Following \cite{Alp:08} we make the following definition:
\begin{definition}
A map $q \colon X \to M$ with $M$ an algebraic space is called a {\em
  good quotient} if the following conditions are satisfied:

(i) The functor $G-{\mathcal Qcoh} X \to {\mathcal Qcoh} M$, ${\mathcal
  F} \mapsto (q_* {\mathcal F})^G$ is exact where $G-{\mathcal Qcoh}
X$ is the category of $G$-linearized quasi-coherent sheaves on $X$.

(ii) $(q_*{\mathcal O}_X)^G = {\mathcal O}_M$.

The map $q \colon X \to M$ is a {\em good geometric quotient} if, in addition, the orbits of closed points are closed.
\end{definition}

\begin{remark} 
Condition (i) implies that the quotient map is affine, and if $G$ is
linearly reductive it is equivalent to the quotient map being
affine. However, if $G$ is not linearly reductive then the quotient
map may be affine without condition (i) being satisfied.
\end{remark}

In \cite{Alp:08} Alper generalized the concept of good quotient to Artin stacks.
\begin{definition} \cite{Alp:08}
Let ${\mathcal X}$ be an Artin stack. A map $q \colon {\mathcal X} \to M$
with $M$ an algebraic space is a {\em good moduli space} for ${\mathcal X}$
if

i) The map $q \colon {\mathcal X} \to M$ is cohomologically affine. That is
the functor $q_* \colon {\mathcal Q}coh\;{\mathcal X} \to {\mathcal Q}\;coh M$ is exact.

ii) $q_* {\mathcal O_{\mathcal X}} = {\mathcal O}_M$.
\end{definition}

Good quotients and moduli spaces enjoy a number of natural
properties. The following theorem of Alper summarizes these
properties, and shows that Alper's notion of good quotient is
equivalent to other definitions in the literature.
\begin{theorem} \label{thm.alp} \cite[Theorem 4.16, Theorem 6.6]{Alp:08}
If $q\colon {\mathcal X} \to M$ is the good moduli space of an Artin stack then

(i) $q$ is surjective and universally closed.

(ii) If $Z_1, Z_2$ are closed substacks then $q(Z_1 \cap Z_2) = q(Z_1)
\cap q(Z_2)$. In particular the images of two disjoint closed
substacks are disjoint.

(iii) $q$ is a universal categorical quotient in the category of
algebraic spaces.
\end{theorem}

\begin{remark}
If $G$ acts on $X$ and there exists a good quotient $X \to M$, then with Alper's definition, the induced map $[X/G] \to M$ is the good moduli space of the quotient stack $[X/G]$. Theorem \ref{thm.alp}
implies that for a linearly reductive group $G$, Alper's notion of a good quotient is equivalent the definition
given by Seshadri \cite[Definition 1.5]{Ses:72}.
\end{remark}

\subsection{Geometric invariant theory}
Let $G$ be a reductive group acting on scheme $X$ and let $L$ be a $G$-linearized line bundle on $X$.
\begin{definition} \cite[Definition 1.7]{MFK:94}
(i) A point $x \in X$ is {\em semi-stable} (with respect to $L$) if 
there is a section $s \in H^0(X,L^n)^G$ such that $s(x) \neq 0$
and $X_s$ is affine.

(ii) A point $x \in X$ is {\em stable} if it is
semi-stable, has finite stabilizer, and the action of $G$ on $X_s$ is closed.
\end{definition}
\begin{remark}
If $L$ is assumed to be ample then the condition that $X_s$ be affine is automatic.

Our use of the term stable differs slightly from the terminology used
in \cite{MFK:94} in that we require that the stabilizer of a stable
point be finite. In \cite{MFK:94} such points are referred to as
properly stable. 

Note that there can be strictly semi-stable points
which have finite stabilizer. For example, consider
the $\CC^*$ action on $\A^4$ given by $t\cdot(x,y,z,w) = (tx,t^{-1}y,
tz,t^{-1}w)$. Since any constant function is $\CC^*$-invariant, all
points are semi-stable with respect to the trivial
linearization, and every point except the origin has trivial stabilizer.
However, the stable locus is the complement of the
linear subspaces $V(x,z)$ and $V(y,w)$. We will mention other aspects of this action in Example \ref{ex.stack3}. 
\end{remark}

The main result of geometric invariant theory can be stated in the
language of good quotients as:
\begin{theorem} \cite{MFK:94}
Let $G$ be a reductive algebraic group acting on a scheme $X$ and let
$L$ be a line bundle on $X$ linearized with respect to the action of
$G$. If $X^{ss}$ and $X^s$ denote the open subsets of semi-stable and
stable points respectively then a good quotient $X^{ss}/G$ exists as a
quasi-projective scheme and contains as an open set a good geometric
quotient $X^{s}/G$. If $X$ is complete and $L$ is ample
then the quotient $X^{ss}/G$
is also projective.
\end{theorem}

\subsection{Reichstein transforms} \label{subsec.reichstein}
Let $G$ be an algebraic group acting on a scheme (or algebraic space)
$X$ and let $q \colon X \to M$ be a good quotient.
\begin{definition} \label{def.reich} Let $C$ be a closed $G$-invariant
  subscheme of $X$ and let $\tilde{C} = q^{-1}(q(C))$ be the
  saturation of $C$ relative to the quotient map.  The {\em Reichstein
    transform}, ${\mathcal R}(X,C)$ of $X$ relative to $C$ is defined
  to $(\Bl_C X) \smallsetminus (\tilde C)'$ where $\tilde{C}'$ is the
  strict transform of $\tilde{C}$ in the blow-up $\Bl_C X$ of $X$
  along $C$.

More generally if ${\mathcal X}$ is an Artin stack with good moduli
space
$q \colon {\mathcal X} \to M$ and ${\mathcal C} \subset \mc{X}$ is a closed substack
then we define the Reichstein transform ${\mathcal R}(\mc{X}, \mc{C})$
to be $(\Bl_{{\mathcal C}} {\mathcal X}) \smallsetminus \tilde{{\mathcal
    C}}'$
where $\tilde{{\mathcal C}} = q^{-1}(q({\mathcal C}))$ and $'$
indicates strict transform.

If $C$ (resp. ${\mathcal C}$) is a Cartier divisor then ${\mathcal R}(X,C)$
(resp. ${\mathcal R}({\mathcal X}, {\mathcal C})$) is called a {\em divisorial
Reichstein transform}. 
\end{definition}
\begin{remark}
A closed subvariety $C \subset X$ is saturated if $\tilde{C} = C$. If
$C$ is saturated then the Reichstein transform is just the 
blow-up $\Bl_C X$. When $q \colon X \to M$ is a good geometric quotient then
every closed subset is saturated. At the other extreme, if $\tilde{C}
= X$ then ${\mathcal R}(X,C)$ is empty.
\end{remark}
\begin{remark}
If $C$ (resp. ${\mathcal C}$) is a Cartier divisor then ${\mathcal R}(X,C)$
(resp. ${\mathcal R}({\mathcal X}, {\mathcal C})$) is open  in $X$
(resp. ${\mathcal X}$). Indeed it is the complement of the closure of
 $\tilde{C}\smallsetminus C $ (resp. the closure of 
$\tilde{{\mathcal C}}\smallsetminus {\mathcal C}$).
\end{remark}

The reason for our terminology comes from a theorem of Reichstein in
geometric invariant theory.
\begin{theorem}(cf. \cite[ Theorem 2.4]{Rei:89})
Let $X$ be a smooth projective variety and let $C \subset X$
be a smooth subvariety. Let $L$ be a $G$-linearized line bundle and
let $X^{ss}$ be the set of $L$-semi-stable points. Then, for a suitable
linearization of the $G$-action on $\Bl_C X$, we have $(\Bl_C X)^{ss} = {\mathcal R}(X^{ss},
C \cap X^{ss})$. In particular, a good quotient
of ${\mathcal R}(X^{ss}, C \cap X^{ss})/G$ exists as projective variety.
\end{theorem}

\section{Toric stacks}
The goal of this section is to define toric stacks. Toric stacks were
originally 
defined in \cite{BCS:05} using {\em stacky fans}. Iwanari \cite{Iwa:09}
and subsequently Fantechi, Mann and Nironi \cite{FMN:07}
gave an  intrinsic definition of Deligne-Mumford toric stacks and
showed the relationship between simplicial stacky fans and toric
Deligne-Mumford stacks. (Note that the toric variety does not uniquely
determine a stacky fan). In this article we are mainly interested in
toric Artin stacks which we define to be Artin stacks associated to possibly
non-simplicial stacky fans. We do not give an intrinsic definition in
this case. Our definition of stacky fan  is less general than the one given
by Satriano in \cite{Sat:09}.

\subsection{Toric notation}\label{section.toric}
We establish the notation that we will use for toric varieties. A
basic reference is the book by Fulton \cite{Ful:93}. Let $N$ be a
lattice of rank $r$ and let $M = \Hom(N,\ZZ)$ be the dual lattice.
Given a fan $\Sigma \subset N_{\RR} = \RR^r$ we denote the associated
toric variety by 
$X= X(\Sigma)$. Let
$\Sigma(1)$ denote the rays in $\Sigma$. We assume throughout that $\Sigma(1)$ spans $N_\RR$.
Given a cone $\s \in \Sigma$, we denote by $X_\s \subseteq X$ the
affine open set associated to $\s$. By definition $X_\sigma =
\Spec\CC[\sigma^\vee \cap M]$ and the open sets $X_\s$ form an affine cover
of the toric variety $X(\Sigma)$. Likewise let $\gamma_\s$ be the
distinguished point of $X_\s$ corresponding to the irrelevant ideal of
the $M$ graded ring $\CC[\sigma^\vee \cap M]$. The notation $V(\s)$ described next will come up often.
\begin{definition}\label{def.vs}
If $\s$ is a cone let $O(\s)$ denote the orbit of $\gamma_\s$ under the torus $T_N$ of $X$, and let $V(\s)=\overline{O(\s)}$ be its closure in $X$.
\end{definition}

The orbit-cone correspondence states 
\begin{equation}\label{eq.orbit-cone-corr}
V(\s)=\coprod_{\tau\ contains\ \s\ as\ a\  face} O(\tau).
\end{equation}

We also recall some facts about toric morphisms. Let $N'$ be a lattice
and let $\Sigma'$ be a fan in ${N'}_\RR$.  A map of lattices
$\overline{\phi}:N' \to N$ is said to be compatible with $\Sigma'$ and
$\Sigma$ if for every cone $\s'\in \Sigma'$, there is a cone $\s\in
\Sigma$ such that $\overline{\phi}_\RR(\s') \subseteq \s$ (where
$\overline{\phi}_\RR: (N')_\RR \to N_\RR$ is the map induced by
$\overline{\phi}$). Such a compatible map $\overline{\phi}$ induces a
toric morphism $\phi:X(\Sigma') \to X(\Sigma)$. For $\s' \in
\Sigma'$, if $\s$ is the minimal cone in $\Sigma$ such that
$\overline{\phi}_\RR(\s') \subseteq \s$, then
$\phi(\gamma_{\s'})=\gamma_\s$ and $\phi(O(\s'))\subseteq O(\s)$
(\cite[Lemma 3.3.21a]{CLS:10}).

Suppose $\overline{\phi}$ has finite cokernel, or equivalently
$\phi_\RR$ is surjective. Then $\overline{\phi}\otimes\CC^* :N' \otimes
\CC^* \to N \otimes \CC^* $ is surjective, and hence
$\phi\vert_{T_{N'}}:T_{N'} \to T_N$ is surjective. In this case,
$\phi(O(\s'))=\phi(T_{N'}\cdot\gamma_{\s'})=T_N\cdot\gamma_\s=O(\s)$.

\subsection{Cox construction of toric varieties as
  quotients}\label{cox-construction}
As a first step in our discussion of toric stacks we recall Cox's
\cite{Cox:95} or \cite[Thm. 5.1.11]{CLS:10} description of $X(\Sigma)$ as
a quotient of an open subset of $\A^n$ by a diagonalizable group.  

\begin{definition}[Cox construction of $X(\Sigma)$]\label{def.cox-construction}
  Given a fan $\Sigma \subseteq N$, let $\widetilde{N}=\ZZ^n$, where
  $n = | \Sigma(1)|$ and let $\{e_\rho \mid \rho \in \Sigma(1)\}$ be
  the standard basis of $\widetilde{N}$. For each cone $\s \in
  \Sigma$, define $\widetilde{\s}=\Cone(e_\rho \mid \rho\in\s(1))
  \subseteq \RR^{n}$. These cones form a fan
  $\widetilde{\Sigma}=\{\widetilde{\s} \mid \s \in \Sigma \}$. Define
  $\overline{q}:\widetilde{N} \to N$ by $\overline{q}(e_\rho)=u_\rho$,
  where $u_\rho$ is the primitive vector along the ray $\rho \in
  \Sigma(1)$. Then $\overline{q}$ is compatible with
  $\widetilde{\Sigma}$ and $\Sigma$, and so induces a toric morphism
  $q:X(\widetilde{\Sigma})\to X(\Sigma)$, called the Cox construction
  of $X(\Sigma)$. Cox \cite[Thm 2.1]{Cox:95} showed that $q$ is a good
  quotient for the group $G_\Sigma = \ker(T_{\widetilde{N}} \to T_N)$, and if $\Sigma$ is a simplicial fan, then $q$ is a good
  geometric quotient.
\end{definition}

A more geometric description of $X(\widetilde{\Sigma})$ is as follows. The homogeneous coordinate ring of $X(\Sigma)$ is the ring $S=\CC[x_\rho \mid \rho \in \Sigma(1)]$. For a cone $\sigma\in
\Sigma$, let $x^{\hat{\sigma}}=\Pi_{\rho \notin \sigma(1)} x_\rho$ be
the corresponding monomial. Let $B=B_\Sigma \subset S$ be the ideal of $S$
generated by the $x^{\hat{\sigma}}$, i.e. $B=\langle x^{\hat{\sigma}}
\mid \sigma \in \Sigma \rangle$. Let $Z_\Sigma \subset
\A^{n}$ be the subvariety $\mathbb{V}(B)$ associated to $B$.
By \cite[Proposition 5.1.9a]{CLS:10} we have
that $X(\widetilde{\Sigma})=\A^{n}\smallsetminus Z_\Sigma.$

Define a map 
$M \to \ZZ^{n}$ by $m \mapsto \left(\langle m,
  u_\rho \rangle\right)_{\rho \in \Sigma(1)}$ where $u_\rho$ is a
primitive generator for the ray $\rho$. This map induces a short exact sequence of diagonalizable groups
\begin{equation} \label{seq.diag}
1 \to G_\Sigma \to (\CC^*)^{\Sigma(1)} \to T \to 1
\end{equation}
$T=\Hom(M, \CC^*)$ is the torus acting on the toric variety $X(\Sigma)$.

Since $(\CC^*)^{\Sigma(1)}$ acts naturally on $\A^n$, we have an induced action
of $G_\Sigma$  on $\A^{n}$.

\begin{theorem}\cite[Thm 2.1]{Cox:95}\label{cox}
The toric variety $X(\Sigma)$ is naturally isomorphic to the
good quotient of $X(\widetilde{\Sigma}) = \A^{n} \smallsetminus
Z_\Sigma$ by $G_\Sigma$. 
Also, this quotient is a geometric quotient if and only if $\Sigma$ is simplicial.
\end{theorem}

\begin{definition}
  We refer to $X(\widetilde{\Sigma}) = \A^{n} \smallsetminus Z_\Sigma$
  as the {\em Cox space} of the toric variety $X(\Sigma)$. We denote
  the corresponding toric quotient stack
  $[X(\widetilde{\Sigma})/G_\Sigma]$ as ${\mathcal X}(\Sigma)$ and
    refer to it as the {\em Cox stack} of the toric variety
    $X(\Sigma)$.
\end{definition}

\begin{proposition}
The toric variety $X(\Sigma)$ is a good moduli space (in the sense of
\cite{Alp:08}) for the quotient stack ${\mathcal X}(\Sigma)$.
\end{proposition}
\begin{proof}
This can be checked locally on ${\mathcal X}(\Sigma)$. The open set
$X_\sigma \in {\mathcal X}(\Sigma)$ is the quotient of an affine open
subset of $\A^{\Sigma(1)}$ by the action of the linearly reductive group
$G_\Sigma$ and the result follows from \cite{Alp:08}.
\end{proof}

\begin{remark} If the toric variety  $X(\Sigma)$ is not projective
  or quasi-projective then ${\mathcal X}(\Sigma)$ is an example of an
  Artin stack with a good moduli space which is not a quotient stack arising from GIT, i.e. it is not
  of the form $[X^{ss}/G]$, where $X^{ss}$ is the set of semi-stable points for some linearization on a projective variety $X$.
\end{remark}

\subsection{Toric stacks}
Given a fan $\Sigma$, the Cox construction can be generalized to
produce other quotient stacks whose good moduli spaces are the toric
variety $X(\Sigma)$. Such stacks were defined by Borisov, Chen and
Smith in \cite{BCS:05} and were called {\em toric stacks}\footnote{It
  should be noted that Lafforgue also used the term toric stack. For
  Lafforgue a toric stack is the quotient stack $[X(\Sigma)/T_N]$ and
  is not the same as those constructed in \cite{BCS:05}.}. To give
Borisov, Chen and Smith's definition of toric stacks we begin by
defining the notion of a stacky fan.

\begin{definition}
A {\em stacky fan}
${\mathbf\Sigma}=(N, \Sigma, (v_1, \ldots, v_{n}))$ consists of:
\begin{list}{\arabic{item-counter}.}{\usecounter{item-counter}}
\item  a finitely generated abelian group $N$
of rank $d$
\item a (not necessarily simplicial) fan $\Sigma \subset N_\QQ$, whose rays we will denote by $\rho_1, \ldots, \rho_n$,
\item for every $1 \leq i \leq n$, an element $v_i \in N$ such that its image $\overline{v_i} \in N_\QQ$ lies on $\rho_i$.
\end{list}
\end{definition}
The data of a tuple $(v_1, \ldots, v_n)$ is equivalent to a
homomorphism $\map{\beta}{\ZZ^{n}}{N}$ such that
$\beta(e_i)=v_i$. Hence sometimes the stacky fan ${\mathbf\Sigma}$ is
written as ${\mathbf\Sigma}=(N, \Sigma, \beta)$.

Let $\DG(\beta)$ be the Gale dual group of $N$ (as defined in
\cite{BCS:05}) and let $G_{{\mathbf \Sigma}} =
\Hom(\DG(\beta),\CC^*)$. There is an induced map of abelian groups
$\beta^\vee\colon (\ZZ^n)^* \to \DG(\beta)$. Applying the functor $\Hom(\_,
\CC^*)$ to the homomorphism $\beta^\vee$ yields a map $G_{\mathbf
  \Sigma} \to (\CC^*)^n$ thereby defining an action of $G_{\mathbf
  \Sigma}$ on $\Af^n$.
\begin{definition}
  The toric stack ${\mathcal X}({\mathbf \Sigma})$ is defined to be
  the quotient stack $[X(\widetilde{\Sigma})/G_{\mathbf \Sigma}]$ where
as before  $X(\widetilde{\Sigma})=\A^n \smallsetminus Z_{\Sigma}$.
\end{definition}

\begin{remark} Note that while the stack ${\mathcal X}({\mathbf
    \Sigma})$ depends on the full stacky fan, the Cox space
  $X(\tilde{\Sigma})$ depends only on the underlying fan $\Sigma$.
\end{remark}

To understand the relationship between the Cox construction and toric
stacks we give (following \cite{BCS:05}) an explicit presentation for
the dual group $DG(\beta)$ and corresponding diagonalizable group
$G_{{\mathbf \Sigma}}$. This will allow us to see that when $N$ is free and
the $b_i$ are minimal generators for the rays in the underlying fan $\Sigma$
then $G_{{\mathbf \Sigma}} = G_{\Sigma}$.

Since $N$ is a finitely generated abelian group of rank $d$ it has a
projective resolution as a $\ZZ$-module of the form
 $\xymatrix{0 \ar[r] &\ZZ^{r} \ar[r]^{[Q]} &\ZZ^{d+r} \ar[r] & N
  \ar[r] & 0}$
where $Q$ is a  $(d+r) \times r$ integer valued matrix. The map
$\beta \colon \ZZ^n \to N$ lifts to a 
map $\ZZ^{n} \to \ZZ^{d+r}$ given by an integer matrix $B$.
Since we have two maps with target $\ZZ^{d+r}$ we obtain a map
$\ZZ^{n+r} \stackrel{[BQ]} \to \ZZ^{d+r}.$ Borisov, Chen and Smith show
that the dual group
$\DG(\beta)$ may be identified with the cokernel of the dual map
$(\ZZ^{d+r})^* \stackrel{[BQ]^*} \to (\ZZ^{n+r})^*$.

\begin{proposition}
If $N$ is free and $\beta_1= \rho_1, \ldots , \beta_n = \rho_n$ 
are generators for the
rays of $\Sigma$ then $G_{{\mathbf \Sigma}} = G_\Sigma$.
\end{proposition}
\begin{proof}
  If $N$ is free then we can take $Q$ to be the 0 matrix, since $N =
  (\ZZ)^d$ and $\beta = B$ as maps $\ZZ^n \to \ZZ^d$. The dual map is
  the map $M = (\ZZ^d)^* \to (\ZZ^n)^*$ given by $m \mapsto (\langle
  m,\rho_1 \rangle, \langle m, \rho_2 \rangle, \ldots \langle m ,
  \rho_n\rangle)$.  The cokernel of this map is the Weil divisor class group
  $A_{n-1}(X(\Sigma))$, and taking duals gives equation \eqref{seq.diag} defining $G_\Sigma$.
\end{proof}

If $\Sigma$ is simplicial, then $\mathcal{X}(\mathbf{\Sigma})$ is a Deligne-Mumford stack.  
\begin{proposition}
If ${\mathbf \Sigma}$ is a stacky fan then 
the toric variety $X(\Sigma)$ is the good moduli space of the stack 
${\mathcal X}({\mathbf \Sigma})$. If $\Sigma$ is simplicial then $X(\Sigma)$ is the coarse moduli space of
${\mathcal X}({\mathbf \Sigma})$.
\end{proposition}

\subsection{The non-simplicial index of fan}
To analyze toric stacks (or toric varieties) when $\Sigma$ is non-simplicial, we introduce and prove some basic properties of what we call the {\em non-simplicial index} of a strongly convex polyhedral cone, or more generally a fan.
\begin{definition}
Let $\sigma \subset \RR^n$ be a strongly convex rational 
polyhedral cone. Define the {\em
  non-simplicial index of} $\s$ to be $\ns(\sigma) = |\sigma(1)| - d_\sigma$ where
$|\sigma(1)|$ is the number of rays in $\sigma$ and $d_\sigma=\dim \sigma$ is the
dimension of the linear subspace spanned by $\sigma$.
\end{definition}
\begin{remark}
With this definition, a cone is simplicial if and only if $\ns(\sigma) = 0$.
\end{remark}
\begin{lemma} \label{lemma.ns1}
If $\tau$ is a face of a cone $\sigma$ then $\ns(\sigma) \geq \ns(\tau)$.
\end{lemma}
\begin{proof}
The statement is trivial if $\sigma=\tau$, so assume $\sigma \neq \tau$.
By induction we may assume that $\tau$ is a facet (i.e. a codimension one face). 
Then $\sigma$ is spanned by at least $|\tau(1)| +1$ rays. Since $d_\sigma = d_\tau + 1$, it follows that $\ns(\sigma) \geq |\tau(1)|+1 - d_\sigma = |\tau(1)| - d_\tau=\ns(\tau)$.
\end{proof}
The next Lemma is slightly more subtle. 
\begin{lemma} \label{lemma.ns2}
If $\ns(\sigma) > 0$ then there is a unique minimal face $\tau \subset \sigma$
such that $\ns(\tau) = \ns(\sigma)$.
\end{lemma}
\begin{proof}
Let $k=\ns(\s)$. We give a proof by contradiction. Suppose that $\tau_1$ and $\tau_2$
are two distinct minimal faces of $\s$ with $\ns(\tau_1) = \ns(\tau_2)=k$. By
assumption, $\ns(\tau_1 \cap \tau_2)<k$ and thus $\tau_1 \cap \tau_2$ is spanned by fewer than $k+\dim (\tau_1 \cap \tau_2)$ rays. Hence $|\tau_1(1) \cup
\tau_2(1)| > |\tau_1(1)| + |\tau_2(1)|- k-\dim(\tau_1 \cap \tau_2)$.  Let $d_{12}$ is the
dimension of linear subspace of $\RR^n$ spanned by $\tau_1$ and
$\tau_2$. We have $d_{12} = d_{\tau_1} + d_{\tau_2} - \dim(\tau_1 \cap \tau_2)$. Then 
\begin{align*}
\s(1) & \geq |\tau_1(1) \cup \tau_2(2)| + d_\s - d_{12} \\
& > |\tau_1(1)| + |\tau_2(1)|- k-\dim(\tau_1 \cap \tau_2) + d_\s - d_{12} \\
& = |\tau_1(1)| + |\tau_2(1)|- k-\dim(\tau_1 \cap \tau_2) + d_\s - (d_{\tau_1} + d_{\tau_2} - \dim(\tau_1 \cap \tau_2)) \\
& = k + d_\s
\end{align*}
where we used $|\tau_i(1)| = \ns(\tau_i) + d_{\tau_i}=k + d_{\tau_i}$ to get the last line. 
But $\s(1) > k + d_\s$ contradicts $k=\ns(\s)=\s(1)- d_\s$.
\end{proof}

\begin{definition}
If $\sigma$ is a non-simplical cone then the minimal face $\tau \subset \sigma$
such that $\ns(\tau) = \ns(\sigma)$ is called the {\em minimal non-simplicial face} of $\sigma$.
\end{definition}

More generally we can define the non-simplicial index of a fan.
\begin{definition} \label{def.nonsimplicial}
Let $\Sigma \subset \RR^n$ be a fan. The non-simplicial index of $\Sigma$, denoted by
$\ns(\Sigma)$, is defined to be $\max_{\sigma \in \Sigma}\{\ns(\sigma)\}$.
A minimal cone $\sigma \in \Sigma$ with $\ns(\sigma) = \ns(\Sigma)$ is a called
a {\em minimal non-simplicial cone} of $\Sigma$.
\end{definition}

The following is an immediate consequence of Lemma \ref{lemma.ns2}.
\begin{lemma} \label{lemma.ns3} No two minimal non-simplicial cones of
  $\Sigma$ are contained in a cone of $\Sigma$.
\end{lemma}

\subsection{The fibers of the quotient map $q:X(\widetilde{\Sigma}) \to X(\Sigma)$}

The following variety $L_\s$ will be used throughout the paper. 

\begin{definition}[$L_\s$]\label{def.l}
Let $L_\sigma$ be the orbit closure
$V(\widetilde{\s})=\VV(\{x_\rho\}_{\rho \in \sigma(1)}) \cap
X(\tilde{\Sigma}).$ The subspace $L_\sigma$ is clearly
$(\CC^*)^n$-invariant.
\end{definition}

\begin{lemma}\label{stabilizer-size}
The $G_{\mathbf \Sigma}$-stabilizer at a general point of $L_\sigma$
has rank $\ns(\sigma)$ where $\ns(\sigma)$ is the non-simplicial index
of the cone $\sigma$ (Definition \ref{def.nonsimplicial}).
\end{lemma}
\begin{proof}
  The $(\CC^*)^{\Sigma(1)}$-stabilizer of a general point of
  $L_\sigma$ is the torus $(\CC^*)^{l}$ where $l = | \sigma(1)|$. The $G_{\bf
    \Sigma}$-stabilizer is the group $G_{{\bf \Sigma}, \sigma}
  :=\Hom(\DG(\beta)_\sigma,\CC^*)$ where $\DG(\beta)_\sigma :=
(\ZZ^l)^* \otimes_{(\ZZ^n)^*} \otimes \DG(\beta)$
and the map $(\ZZ^n)^* \to (\ZZ^l)^*$ is dual to the
inclusion map $\ZZ^l \to \ZZ^n$. Thus to compute the rank of
the stabilizer $G_{{\mathbf \Sigma},\sigma}$ we need to compute the
rank of the abelian group $\DG(\beta)_\sigma$.  Let $N' =
N/T$ where $T$ is the torsion subgroup of $N$ and let $\beta'$ be the
composite of $\beta$ with the projection $N \to N'$.  As noted in the
proof of \cite[Proposition 5.4]{Sat:09} there is an injection of
abelian groups $\DG(\beta') \to \DG(\beta)$. Since both groups have
the same rank the map also has maximal rank.

Thus, to compute the rank of $\DG(\beta)_\sigma$
it suffices to compute the rank of the group $\DG(\beta')_\sigma :=
\DG(\beta') \otimes_{(\ZZ^n)^*} (\ZZ^{l})^*$. The map
$\beta'$ maps standard basis vectors in $\RR^n$ to 
rational generators of the cones of the fan $\Sigma$. Thus the rank of
$\DG(\beta')_\sigma$ is $|\sigma(1)| - \dim \sigma = \ns(\sigma)$.
\end{proof}

Let $q \colon X(\widetilde{\Sigma}) \to X(\Sigma)$ be the quotient map. For the next lemma, recall that $V(\s)$ was introduced in Definition \ref{def.vs} as the closure of the orbit of the point $\gamma_\s$.

\begin{lemma} \label{lemma.imagelsigma}
If $\sigma$ is a cone of $\Sigma$ then
$${q(L_\sigma)} = V(\sigma).$$
\end{lemma}

\begin{proof}
  Again let $N'$ be the quotient $N/T$ where $T$ is the torsion
  subgroup and let ${\mathbf \Sigma}' = (N', \beta', \Sigma)$ be the
  associated stacky fan.  There is a finite morphism ${\mathcal X}({\mathbf
  \Sigma}) = [X(\widetilde{\Sigma})/G_{{\mathbf \Sigma}}] \to
{\mathcal X}({\mathbf \Sigma}') = [X(\widetilde{\Sigma})/G_{{\mathbf \Sigma}'}]$ and both stacks have $X(\Sigma)$
  as their good moduli spaces. Thus it suffices prove the Lemma for
  the quotient by $G_{{\mathbf \Sigma'}}$. As a consequence we may
  assume that the group $N$ is free abelian.

The map $q$ is constructed by patching quotient maps $q_\tau: U_\tau \to
X_\tau$ for each cone $\tau$ 
of the fan. Here $X_\tau = \Spec \CC[ \tau^\vee \cap M]$ 
is the affine toric variety corresponding to the cone $\tau$
and $U_\tau$ is the affine open set  $\A^{n} \smallsetminus
\VV(x^{\hat{\tau}})$. The map $q_\tau$ is induced by the map
of rings $\CC[\tau^\vee \cap M] \to 
S_\tau=\CC[\{x_\rho\}_{\rho \in \Sigma(1)}]_{x^{\hat{\tau}}}$, given by $\chi^m \mapsto x^{D_m}$ where 
$x^{D_m} = \prod_{\rho} x_\rho^{\langle m, b_\rho \rangle}$.
Given a cone $\sigma$ 
the linear space $L_\sigma$ has non-empty
intersection with the open set $U_\sigma$, and since $q(L_\sigma)$ is closed (because $q$ is a good categorical quotient), it suffices to prove that
the generic point of $L_\sigma$ maps to the generic point of
$V(\sigma) \subset X_\sigma$. The generic point of $L_\sigma$ is the
ideal $I_\sigma = \langle \{x_\rho\}_{\rho \in \sigma(1)}\rangle$. Its
inverse image under the ring homomorphism $\chi^m \mapsto x^{D_m}$ is the
ideal generated by $\{\chi^m| \langle m, \rho\rangle\ >0\}_{\rho \in
  \sigma(1)}$ But this is exactly the ideal of $V(\sigma) \cap X_\sigma =
\Spec \CC[\sigma^\perp \cap (\sigma^\vee \cap M)]$ in $X_\sigma=
\Spec \CC[\sigma^\vee \cap M]$.
\end{proof}

\begin{lemma} \label{lemma.key-inclusion}
Let $J \subseteq \Sigma(1)$ and let $\tau$ be the minimal cone (if
such exists) of
$\Sigma$
which contains $J$.
Then
$q(\VV(\{x_\rho \}_{\rho \in J} )) \subseteq V(\tau).$
\end{lemma}
\begin{remark} If $J$ does not lie in any cone of $\Sigma$
then $\VV(\{x_\rho\}_{\rho \in J}) \cap \tilde{X(\Sigma)}$ is empty.
\end{remark}
\begin{proof}
By Lemma  \ref{lemma.imagelsigma}
we know that  $q( \VV(x_\rho)) \subseteq V(\rho)$. Hence
\begin{eqnarray*}
q(\VV(\{x_\rho\}_{\rho \in J} ))& =&q(\cap_{\rho \in J} \VV(x_\rho ))\\ 
& &\subseteq \cap_{\rho \in J} q( \VV(x_\rho))\\
& &\subseteq \cap_{\rho \in J} V(\rho)= V(\tau) 
\end{eqnarray*}
where the equality in the last line follows from the discussion
on pp.99-100 of  \cite{Ful:93}.
\end{proof}

\begin{definition}\label{def.tilde-ls}
Let $\sigma$ be a cone in the fan $\Sigma$. 
Define  $\tilde{L}_\sigma = q^{-1}(V(\sigma))=q^{-1}(q(L_\s))$.
\end{definition}
 
\begin{definition} \label{def.itilde}
Let $\mu$ be a cone containing $\sigma$ and let
$\tilde{I}_{\mu,\sigma}$ be the ideal generated by 
the monomials
$x^{\hat{\tau}}_\mu :=\prod_{\rho \in \mu(1) \smallsetminus \tau(1)} x_\rho$ 
where
$\tau$ runs over all proper faces of $\mu$
such that 
that $\tau$ does not contain $\sigma$; i.e. $\tau \cap \sigma $ is a proper face of $\sigma$.
Let $\tilde{I}_\sigma = \cap_{\mu \supset \sigma} \tilde{I}_{\mu,\sigma}$.
\end{definition}

\begin{remark} \label{rem.maximal}
Observe that if $\mu' \subset \mu$ then  $\tilde{I}_{\mu',\sigma} 
\supset \tilde{I}_{\mu,\sigma}$ so we may assume that the cones $\mu$ in
Definition \ref{def.itilde} are maximal.
\end{remark}

\begin{proposition} \label{prop.ltilde}
$\tilde{L}_\sigma = \VV(\tilde{I}_\sigma)$.
\end{proposition}
\begin{proof}
It suffices to prove the statement for each open in the covering
$U_\mu$ of $X(\widetilde{\Sigma})$ where
$\mu$ is a (maximal) cone.
Note that if $\mu$ is a cone then $U_\mu \cap \tilde{L}_\sigma =\emptyset$   
unless $\sigma \subset \mu$ (because $p \in \tilde{L}_\sigma \cap U_\mu$ implies $q(p) \in V(\sigma) \cap X_\mu$, which is empty unless $\sigma \subset \mu$). We also claim that $U_\mu \cap \VV(\tilde{I}_\sigma) = \emptyset$ if $\mu$ does not contain $\sigma$. To see this note that if $\mu$ does not contain $\s$ and $\mu'$ is a cone containing $\sigma$
then $\tau' =\mu' \cap \mu$ is a face of $\mu'$ not containing $\sigma$
and $x^{\hat{\tau'}}_{\mu'} \neq 0$ on $U_\mu$.

Thus it suffices to prove the proposition 
for the quotient map $U_\mu \to X_\mu$ for any (maximal) cone $\mu$
containing $\sigma$. On $U_\mu$, $x_\rho \neq 0$ if $\rho \notin \mu$
so we may assume that ${\tilde I}_\sigma$ is generated
by the products $x^{\hat{\tau}}$ where $\tau$ runs over all faces
of $\mu$ that do not contain $\sigma$.

Step I: $\VV(\tilde{I}_\sigma) \supset \tilde{L}_\sigma$.

To prove Step I we will show that $\sqrt{\tilde{I}_\sigma} \subset
I(\tilde{L}_\sigma)$. Let $x^{\hat{\tau}} \in \tilde{I}_\s$ be a generator, where $\tau$ is a face of $\mu$ that does not contain $\s$. By definition $I(\tilde{L}_\sigma)$ is the radical of the ideal generated by
$x^{D_m}$ for all $m \in (\mu^\vee \cap M)$ such that $\langle m, \rho \rangle > 0$ for some ray $\rho \in \sigma(1)$.
If $\tau \cap \sigma$ is a proper face of $\sigma$ then there is an 
$m \in \tau^\perp \cap (\mu^\vee \cap 
M)$ such that $\langle m, \rho \rangle > 0$
for some ray $\rho \in \sigma(1) \cap \tau(1)^c$. 
Hence, for  
$N>>0$ such that $x^{D_m} | (x^{\hat{\tau}})^N$. Thus,
$\sqrt{\tilde{I}_\sigma} \subset I(\tilde{L}_\sigma)$ as claimed.

Step II: $q(\VV(\tilde{I}_\s)) \subset V(\sigma).$

Since $\tilde{I}_\sigma$ is a monomial ideal the primes in its primary
decomposition are all generated by the $x_\rho$ with $\rho \in
\mu(1)$.  For a subset $J \subseteq \mu(1)$ let ${\mathfrak p}_J$ be
the prime ideal generated by $\{x_\rho\}_{\rho \in J}$.  Let $\nu$ be
the minimal face of $\mu$ containing the rays in $J$. Suppose that $\nu
\supset \sigma$. If $\tau$ is a face of $\mu$ not containing $\sigma$
then there is a ray $\rho \in J$ such that $\rho \notin
\tau(1)$. Hence $x_\rho | x_\mu^{\hat{\tau}}$. Thus ${\mathfrak p}_J
\supset \tilde{I}_\sigma$.  On the other hand if $J$ spans a face
$\nu$ that doesn't contain $\sigma$ then $x_\mu^{\hat{\nu}} \in
\tilde{I}_\sigma$ but $x_\mu^{\hat{\nu}} \notin {\mathfrak p}_J$. Thus
${\mathfrak p}_J$ contains $\tilde{I}_\sigma$ if and only if $J$ spans
a face of $\mu$ which contains $\sigma$.   By Lemma
\ref{lemma.key-inclusion} $q_\sigma(\VV({\mathfrak p}_J)) \subset
V(\sigma)$ if $J$ spans a cone containing $\sigma$. 
This proves Step II and with it Proposition
\ref{prop.ltilde}.
\end{proof}

\begin{lemma} \label{lem.ltilde=l}
$\tilde{L}_\sigma = L_\sigma$ if and only if every cone $\mu$ containing 
$\sigma$ is simplicial.
\end{lemma}
\begin{proof}
If $\mu \supset \sigma$ is simplicial then the quotient map
$q_\mu \colon U_\mu \to X_\mu$ is a geometric quotient. Since
$L_\sigma$ is closed and 
$G_{{\mathbf \Sigma}}$-invariant, it follows that 
$q_\sigma^{-1}(q_\sigma(L_\sigma)) = L_\sigma$.
\end{proof}

\section{Blowups and the stacky star subdivision of a stacky fan}
Let $\mathbf{\Sigma}=(N, \Sigma, (v_1, \ldots, v_n))$ be a stacky fan,
where the fan $\Sigma \in N_\RR$ is not necessarily simplicial. 
\begin{definition}[stacky star subdivision of a stacky fan]\label{def.star}
  Let $\sigma$ be a cone in $\Sigma$. Let $v_0= \sum_{\{k | \in
    \sigma(1)\}} v_{k}$, and let $\rho_0$ be the ray in $N_\RR$ generated by
  $v_0$. Set $\Sigma_\sigma$ to be the fan obtained by replacing
  every cone of $\Sigma$ containing $\sigma$ with the joins of its
  faces with $\rho_\sigma$ (cf. \cite[p. 47]{Ful:93}). The stacky fan $\mbf{\Sigma_\s}=(N, \Sigma_\s, (v_0, v_1, \ldots, v_n))$ will be called the ``star subdivision of $\s$''; for brevity, mention of $\mbf{\Sigma}$ is omitted. Let $\beta_\s:{\ZZ^{n+1}} \to N$ be the
  homomorphism associated to the tuple $(v_0, v_1, \ldots, v_n)$,
  i.e. we have $\beta_\s(e_i)=v_i$ for $0\leq i\leq n$.
\end{definition}

\begin{remark}
Observe that if $\dim \s \geq 2$, the fan $\Sigma_\sigma$ has exactly
$n +1$ rays.  The cones of $\Sigma_\sigma$ can be described
as follows. For each cone $\mu$ of $\Sigma$ not containing $\sigma$
there is a corresponding cone $\mu$ of $\Sigma_\sigma$. Each cone
$\mu$ of $\Sigma$ which contains $\sigma$ is replaced by cones (which we will refer to as the ``new'' cones of $\Sigma_\s$) of the
form $\tau' = \Cone(\tau, \rho_0)$ where $\tau$ is a face of $\mu$
such that $\rho_0 \notin \tau$. Since the cones are convex, $\rho_\s
\notin \tau$ if and only if $\tau$ doesn't contain $\sigma$. Thus the
new cones of $\Sigma_\sigma$ are in bijective correspondence with the
cones of $\Sigma$ which were used in the definition of the ideal
$\tilde{I}_\sigma$.

Also observe that $\sigma$ is replaced by the cones
$\Cone(\tau,\rho_\sigma)$ where $\tau$ runs over the facets of
$\sigma$.
\end{remark}

\begin{lemma}\label{ns-drops}
If $\sigma$ is a minimal non-simplicial cone of $\Sigma$ then every
new cone $\tau'$ of $\Sigma_\sigma$ has $\ns(\tau') < \ns(\sigma)$.
\end{lemma}
\begin{proof}
Because $\sigma$ is a minimal non-simplicial cone, every cone $\mu$ 
of $\Sigma$ containing $\sigma$ is generated by exactly $\ns(\sigma) + \dim \mu$ rays, and hence $\ns(\mu)=\ns(\sigma)$. If $\tau$ is a face of $\mu$ not containing $\rho_0$
then the the uniqueness of minimal
cones (Lemma \ref{lemma.ns2}), implies that $\ns(\tau) < \ns(\sigma).$
The new cone $\tau' = \Cone(\tau,\rho_0)$ has dimension $\dim(\tau) + 1$ and 
is spanned by $\ns(\tau) + \dim \tau + 1$ rays. Therefore 
$\ns(\tau') < \ns(\sigma)$.
\end{proof}

Next we compare the groups acting on $X(\widetilde{\Sigma})$ and
$X(\widetilde{\Sigma_\s})$. 
For notational convenience, we order the elements 
$v_1, \ldots , v_n \in N$ so that $\sigma(1) = \{\rho_1, \ldots , \rho_l\}$, where $\rho_i$ is the ray through $\overline{v_i} \in N_{\RR}$.
Choose coordinates $(t_0, \ldots , t_n)$ on $(\CC^*)^{n+1}$
and $(x_0, \ldots , x_n)$ on $\A^{n+1}$ so that $x_0$ corresponds to
$v_0$
and $x_1, \ldots , x_l$ correspond to the rays in $\sigma(1)$.
Consider the exact sequence
\begin{equation}
1  \to \CC^* \stackrel{\lambda_0} \to (\CC^*)^{n+1} \stackrel{\theta}\to (\CC^*)^{n} \to
1
\end{equation}
where $\lambda_0(t) = (t^{-1}, t, \ldots ,t,1 \ldots 1)$ - with $l$ copies of $t$ and $n-l$ $1$'s -  
and $\theta(t_0, \ldots, t_n) = (t_0t_1, \ldots ,  t_0t_l, t_{l+1}, \ldots
t_n)$.
\begin{lemma}
The map $\theta$ restricts to a map $G_{{\mathbf \Sigma_\sigma}} \to
G_{{\mathbf \Sigma}}$ and induces a short exact sequence
\begin{equation} \label{seq.basicG}
1\to \CC^* \to G_{{\mathbf \Sigma}_\sigma} \to G_{{\mathbf \Sigma}} \to 1
\end{equation}
such that the diagram below commutes.
\begin{equation} \label{diag.diaggps}
\xymatrix{
1 \ar[r] & \CC^* \ar[r]\ar[d]^= & G_{{\mathbf \Sigma}_\sigma}
\ar[r]\ar[d]^{(\beta_\sigma)^\vee} & G_{{\mathbf \Sigma}} \ar[r]\ar[d]^{\beta^\vee}
& 1 \\
1 \ar[r] & \CC^* \ar[r]^{\lambda_0} & (\CC^*)^{n+1} \ar[r]^{\theta} &
(\CC^*)^n \ar[r] & 1
}
\end{equation}
\end{lemma}
\begin{proof}
Consider the short exact sequence of free abelian groups
$$ 0 \to \ZZ \to \ZZ^{n+1} \to \ZZ^{n} \to 0$$
where the first map is given by $k \mapsto (-k,k, \ldots , k, 0,
\ldots 0)$ and the second is given by $(a_0, a_1, \ldots , a_n)
\mapsto
(a_0 + a_1, \ldots , a_0 + a_l, a_{l+1}, \ldots , a_n)$. There is a
commutative diagram of exact sequences.
\begin{equation} \label{diag.betas}
\xymatrix{
0 \ar[r] & \ZZ \ar[r] \ar[d] & (\ZZ)^{n+1} \ar[d]^{\beta_\sigma} \ar[r] &
\ZZ^n \ar[d]^\beta \ar[r] & 0\\
0 \ar[r] & 0 \ar[r] & N \ar[r]^{=} & N \ar[r] & 0
}
\end{equation}
The Gale dual of the map $\ZZ \to 0$ is the identity map $\ZZ^* \to
\ZZ^*$. Moreover, all of the vertical arrows in \eqref{diag.betas}
have finite cokernels. Thus,
by \cite[Lemma 2.3]{BCS:05} we obtain a commutative diagram with exact
rows.
\begin{equation} \label{diag.dgs} 
\xymatrix{0 \ar[r] & (\ZZ^*)^n \ar[r] \ar[d]^{\beta^\vee} &
  (\ZZ^*)^{n+1} 
\ar[d]^{{\beta_\sigma}^\vee}
    \ar[r] &
    \ZZ^* \ar[d]^{=} \ar[r] & 0\\
    0 \ar[r] & \DG(\beta)\ar[r] & \DG(\beta_\sigma) \ar[r] & \ZZ^* \ar[r] & 0
  } 
\end{equation}
Applying the functor $\Hom(\_, \CC^*)$ to \eqref{diag.dgs}
  yields \eqref{diag.diaggps}.
\end{proof}

\begin{definition}[Toric Reichstein transform]\label{def.reich-l}
  Given a stacky fan $\mbf{\Sigma}$ and a cone $\s$ in the underlying fan $\Sigma$, let $Y_\sigma$ be
  the blow-up of $X(\widetilde{\Sigma})=\A^{n} \smallsetminus
  Z_\Sigma$ along the linear subspace $L_\sigma$ and let $Y'_\sigma$
  be Reichstein transform (Definition \ref{def.reich}) of
  $X(\widetilde{\Sigma})$ relative to the linear subspace
  $L_\sigma$. The quotient stack $[Y'_\sigma/G_{{\mathbf
      \Sigma}}]$ is called the {\em toric Reichstein transform} of
  ${\mathcal X}({\mathbf \Sigma})$ relative to the cone $\sigma$, or relative to the substack $[L_\sigma/G_{\mbf{\Sigma}}] \subset {\mathcal X}({\mathbf \Sigma})$.
\end{definition}

\begin{remark}
$Y'_\sigma = Y_\sigma$ if and only if $\tilde{L}_\sigma =  L_\sigma$
which is equivalent (by Lemma \ref{lem.ltilde=l}) to the condition that 
every cone containing $\sigma$ is simplicial.
\end{remark}

We now come to our main technical result. It shows that the toric
Reichstein transform of a toric stack is again a toric stack. As a
result the toric Reichstein transform of a toric stack also has a good
moduli space.  The proof will be given in Section
\ref{subsec.proofofthmtechnical}.

\begin{theorem} \label{thm.technical}
Let $\mathbf{\Sigma} = (N,\Sigma, \beta)$ be a stacky fan, and let
$\s$ be a cone in $\Sigma$. Then the quotient stack
$[Y'_\sigma/G_{\mathbf \Sigma}]$
obtained from the Reichstein transform of ${\mathcal X}({\mathbf
  \Sigma})$
relative to the substack $[L_\sigma/G_{\mathbf \Sigma}]$
is isomorphic to the toric stack $\mc{X}(\mbf{\Sigma_\sigma})$ formed by stacky star subdivision of $\s$ (Definition \ref{def.star}), i.e. there is an isomorphism of stacks ${\mathcal X}(\mbf{\Sigma_\sigma})\simeq
[Y'_\sigma/G_{\mbf{\Sigma}}]$.
\end{theorem}

\begin{remark}
  If one is interested primarily in toric varieties and not in toric
  stacks, the above result says that the toric variety $X(\Sigma_\s)$
  formed by star subdivision of the cone $\s$ is isomorphic to the
  quotient by $G_\Sigma$ of $Y_\s'$, the Reichstein transform of $L_\s
  \subset X(\widetilde{\Sigma})$. However, as we show in Example
  \ref{ex.stack}, the stack language is convenient and natural for
  describing some proper birational morphisms between toric varieties.
\end{remark}

\subsection{Examples}

\begin{example}\label{ex.stack}
  Let $\Sigma \subset \RR^2$ be the fan with a single maximal cone
  generated by $v_1 = e_1$ and $v_2 = e_1 + 2e_2$. 
 Let $\Sigma'$
be the star subdivision of $\Sigma$.  Then $X(\Sigma')$
  is a smooth toric variety with maximal cones $\sigma_1 =
  \Cone(e_1,e_1 + e_2)$ and $\sigma_2 =\Cone(e_1 + e_2, e_1 +
  2e_2)$. There is a map of toric varieties $X(\Sigma') \to X(\Sigma)$
  but there is no map of Cox stacks ${\mathcal X}(\Sigma') \to
  {\mathcal X}(\Sigma)$ because the vector $e_1 + e_2$ is not an
  integral linear combination of $v_1$ and $v_2$. However if we let
  ${\mathbf \Sigma'} = (\ZZ^2, \Sigma', 2(e_1 + e_2), e_1, e_2)$ then
  the map of toric varieties $X(\Sigma') \to X(\Sigma)$ is induced by
  map of toric stacks ${\mathcal X}({\mathbf \Sigma'})\to {\mathcal
    X}(\Sigma)$ corresponding to the Reichstein transform of the stack
  ${\mathcal X}(\Sigma)$ relative to the maximal cone $\sigma$ of
  $\Sigma$.  \end{example}

If ${\mathbf \Sigma} = (N,\Sigma, v_1, \ldots , v_n)$ is a toric fan
then the good moduli space of ${\mathcal X} ({\mathbf \Sigma})$ is the
  toric variety $X(\Sigma)$ and does not depend on the choice of
  elements $v_1, \ldots , v_n \in N$. However, as the next example shows,
the toric variety
 associated to a stacky subdivision relative to a cone $\sigma$
depends on the choice of $v_i$ corresponding to the rays of $\sigma$.
  
\begin{example} \label{ex.stacktwo}
 Let $\Sigma \in \RR^3$ be the fan with a single maximal cone $\s$
  generated by $v_1=e_1, v_2=e_2, v_3=
e_3, v_4=e_1-e_2+e_3$, where the $e_i$ are the
  standard basis vectors and let ${\mathcal X}(\Sigma)$ be the Cox
  stack of $X(\Sigma)$. 
The toric variety $X(\Sigma)$ is the quadric
  cone. Let $\Sigma'$ be the star subdivision of $\sigma$. Then
  $\Sigma'$ is a smooth (hence simplicial) toric variety. Since $v_0 =
  v_1 +v_2 + v_3+ v_4 = 2e_1 + 2e_3$ is not primitive, the  map of
  toric varieties $X(\Sigma') \to X(\Sigma)$ is  induced by a map
of stacks ${\mathcal
    X}({\mathbf \Sigma'}) \to {\mathcal X}(\Sigma)$ where
${\bf \Sigma'} = (\ZZ^3, \Sigma', 2e_1 + 2e_2, e_1,e_2,e_3,e_1 -e_2 + e_3)$
is the stacky star subdivision of the stacky fan ${\bf \Sigma} =
(\ZZ^3, \Sigma, e_1, e_2,e_3, e_1 -e_3 + e_3)$.

Now consider the stacky fan $\mbf{\Sigma}=(\ZZ^3, \Sigma, \{2e_1,
  2e_2, e_3, v_4\})$ 
The underlying toric variety of ${\mathcal X}({\bf \Sigma})$ is also
  $X(\Sigma)$, but the stacky star subdivision of $\mbf{\Sigma}$
relative to $\sigma$  gives the stacky fan $\mbf{\Sigma_\s}=(\ZZ^3, \Sigma_\s,
  \{3e_1+e_2+e_3,2e_1, 2e_2, e_3, v_4\})$, where $\Sigma_\s$ is the fan
  formed by star subdividing $\Sigma$ along the ray through
  $3e_1+e_2+e_3$. We have a commutative diagram

$\xymatrix{
\mc{X}(\mbf{\Sigma_\s})  \ar[r]^{\tilde{f}} \ar[d]_{q(\Sigma_\s)} &\mc{X}(\mbf{\Sigma}) \ar[d]^{q(\Sigma)}\\
X(\Sigma_\s) \ar[r]_{f} &X(\Sigma)
}$

\noindent where $\tilde{f}$ is birational,  $\mc{X}(\mbf{\Sigma_\s})$ is isomorphic to the Reichstein transform of $[L_\s/G_\Sigma] \subset \mc{X}(\mbf{\Sigma})$, and $f$ is a proper birational map between toric varieties. 
\end{example}

Suppose that ${\mathbf \Sigma} = (\Sigma, N, v_1, \ldots , v_n)$ is a
stacky fan.  If $\sigma = \rho$ is a ray in $\Sigma$ then $L_\sigma$
is a Cartier divisor in $X(\widetilde{\Sigma})$ so the blow-up of
$X(\widetilde{\Sigma})$ along $L_\sigma$ is identified with
$X(\widetilde{\Sigma})$. However, the hyperplane $L_\sigma$ need not
be saturated with respect to the quotient map $X(\widetilde{\Sigma})
\to X(\Sigma)$, so the Reichstein transformation of $L_\sigma$ may be
a proper open set in $X(\widetilde{\Sigma})$. As the next example
shows, this phenomenon is related to wall crossing in geometric
invariant theory and will be discussed further in Section
\ref{sec.divreich}.

\begin{example} \label{ex.stack3}
Let $\Sigma$ be the fan in $\RR^3$ considered in Example \ref{ex.stacktwo}. 
Let ${\mathcal X}$ be Cox stack of the toric variety $X(\Sigma)$. Then
${\mathcal X} = [\A^4/\CC^*]$ where $\CC^*$ acts with weights
$(1,-1,1,-1)$. If $L$ is any of the coordinate hyperplanes in $\A^4=\Spec \CC[x_1, x_2, x_3, x_4]$ then $L$ is
not
saturated so the Reichstein transform with respect to $L$ is an open
set in $\A^4$. For example, if $L = L_1 = \VV(x_1)$ then $\tilde{L} =
\VV(x_1) \cup \VV(x_2,x_4)$. Thus the Reichstein transform of
${\mathcal X}(\Sigma)$ with respect to $L$ produces the quotient stack
$[(\A^4 \smallsetminus \VV(x_2,x_4))/\CC^*]$. This is the Cox stack
of the toric variety $X(\Sigma')$ where $\Sigma'$ is the fan obtained
by subdividing $\sigma$ into two cones via the subdivision relative to
the lattice point $v = e_1$. 

This example can be interpreted in terms of change of linearizations
in Geometric Invariant Theory (cf. \cite[Example 1.16]{Tha:96}). The
quotient $\A^4/\CC^*$ is the GIT quotient of $\A^4$ linearized with
respect to the trivial character, while the quotient $(\A^4
\smallsetminus \VV(x_2,x_4))/\CC^*$ is the GIT quotient of $\A^4$
with respect to the character of weight $1$.  Likewise, the GIT quotient
$(\A^4 \smallsetminus \VV(x_1,x_3))/\CC^*$ corresponds to the GIT quotient
where the linearization has weight $-1$. The second quotient corresponds
to the Reichstein transformation with respect to the divisor $L_2 =
\VV(x_2)$.
\end{example}

\subsection{Reichstein transforms and the factorization of birational
  toric morphisms}
The phenomenon of Examples \ref{ex.stack}--\ref{ex.stack3}
can be generalized.
\begin{proposition} \label{prop.subdivision}
Let $\Sigma$ be a fan in a lattice $N$ and let $\Sigma'$ be the subdivision
of $\Sigma$ relative to a ray $\rho_0 \in \vert \Sigma \vert$ in the support of $\Sigma$. Then there exist stacky fans ${\mathbf \Sigma}= (N,\Sigma, v_1,\ldots , v_n)$ and ${\mathbf \Sigma'} = (N, \Sigma', v_0, v_1, \ldots , v_n)$ with $v_0 =
v_1  + \ldots + v_l$ where $v_0$ is a lattice point on $\rho_0$ and $v_1, \ldots , v_l$ are lattice points on the rays of a cone $\sigma \in \Sigma$. Hence for any toric blow-up $X(\Sigma') \to
X(\Sigma)$ there is a toric Reichstein transform of toric  stacks
${\mathcal X}' \to {\mathcal X}$ and a commutative diagram of stacks
and good moduli spaces

$$\xymatrix{
{\mathcal X}'  \ar[r]\ar[d] & {\mathcal X}\ar[d]\\
X(\Sigma') \ar[r] & X(\Sigma)
}$$
\end{proposition}
\begin{proof}
Let $\sigma$ be the unique cone of $\Sigma$ containing $\rho_0$ in its
interior. If $\rho_1, \ldots , \rho_l$ are the rays of $\sigma$ and $w_i$ is the primitive lattice point on
$\rho_i$ then we can write $w_0 = \sum_{i=1}^l a_iw_i$ where the $a_i \in {\mathbb Q}_{>0}$.
Clearing denominators we see that $b_0w_0 = \sum b_i w_i$ with $b_i \in
\ZZ_{>0}$. If we let $v_i = b_i w_i$ for $i = 0, \ldots , l$ and let
$v_{l+1}, \ldots , v_{n}$ be any lattice vectors on the remaining rays
$\rho_{l+1}, \ldots , \rho_n$ of $\Sigma$ then $v_0 = v_1 + \ldots + v_l$.
\end{proof}

Combining Proposition \ref{prop.subdivision} with weak factorization
for birational toric morphisms \cite[Theorem A]{Wlo:97} yields the
following result relating birational toric morphisms and toric
Reichstein transformations.
\begin{corollary} \label{cor.factorization} Let $\Sigma', \Sigma''$ be
  two fans in $\RR^d$ such that $|\Sigma'| = |\Sigma'|$.  Then there
  exists a sequence of stacky fans ${\mathbf \Sigma_i}=(\mathbb{Z},
  \Sigma_i, \beta_i)$, $i=0, \ldots, n$, such that $\Sigma_0=\Sigma'$
  and $\Sigma_n=\Sigma''$ such that the  toric stack ${\mathcal X}({\mathbf
    \Sigma}_i)$ is related to  ${\mathcal X}({\mathbf
    \Sigma}_{i-1})$ by a composition of toric Reichstein transformations
  for $i=1, \ldots, n$.
\end{corollary}
\begin{proof}
  By toric resolution of singularities the fans $\Sigma'$ and $\Sigma''$
  can be subdivided into regular fans. By Proposition
  \ref{prop.subdivision} these subdivisions correspond to Reichstein
  transformations of toric stacks with underlying fans $\Sigma'$ and
  $\Sigma''$. Hence we are reduced to the case that $\Sigma'$ and
  $\Sigma''$ are regular. By \cite[Theorem A]{Wlo:97} there is a sequence of fans
  $\Sigma_i$ $i=0, \ldots, n$, such that $\Sigma_0=\Sigma'$ and $\Sigma_n=\Sigma''$ and
 $\Sigma_i$ is related to $\Sigma_{i-1}$
  via a sequence of subdivisions. Once again using
  Proposition \ref{prop.subdivision} we may interpret these
  subdivisions in terms of toric Reichstein transforms.
\end{proof}

\subsection{Proof of Theorem \ref{thm.technical}} \label{subsec.proofofthmtechnical}
  Label the rays of $\Sigma$ as $\rho_1, \ldots , \rho_{l},
  \rho_{l+1}, \ldots , \rho_{N}$ with $\rho_1, \ldots \rho_l$ the rays
  of $\sigma$. Let $x_1, \ldots , x_n$ be the corresponding
  coordinates on $\A^n$. 
The linear subspace $L_\sigma$ is cut out by the
  regular sequence $(x_1, \ldots , x_l)$ so the blowup $Y_\s$ is
  embedded as the subvariety $X(\widetilde{\Sigma}) \times \Pro^{l-1}
    \subset
\A^n \times \Pro^{l-1}$
defined by the equations $x_iy_j = x_jy_i$ for
  $1 \leq i < j \leq l$.

Likewise label the rays of $\Sigma_\sigma$ as
$\rho_0 , \rho_1, \ldots , \rho_{l}, \rho_{l+1}, \ldots , \rho_N$ with
$\rho_0 = \rho_\sigma$ (introduced in Definition \ref{def.star}). Let $z_0,z_1, \ldots , z_N$ be the
corresponding coordinates on $\A^{n+1}$.

Observe that for each ray $\rho$ of $\s$, there is a cone $\tau'=\Cone(\tau, \rho_s)$ of $\Sigma_\sigma$, such that $\rho$ does not lie in $\tau'$ and such that $\tau$ is a face of $\s$. Hence not all of $z_1, \ldots, z_l$  vanish on
the open set $X(\widetilde{\Sigma_\sigma}) \subset \A^{n+1}$; i.e. 
$X(\widetilde{{\Sigma_\sigma}}) \subset \A^{n+1} \smallsetminus V(z_1, \ldots , z_l)$.

Define a map
$p \colon \A^{n+1} \smallsetminus V(z_1, \ldots , z_l) \to \A^n \times \Pro^{l-1}$
by the formula
$$(z_0, \ldots , z_n)
\mapsto
(z_0z_1, z_0z_2, \ldots, z_0z_l,z_{l+1}, \ldots , z_n) \times
[z_1: \ldots : z_l].$$
By definition the image of $p$ satisfies the equations $x_iy_j= x_j y_i$ for $1 \leq i < j \leq l$
so it is contained in the blowup $\Bl_{L_\sigma}\A^n$. 
Since not all of $z_1, \ldots , z_l$ vanish
on $\A^{n+1} \smallsetminus V(z_1, \ldots , z_l)$ the action of
$\lambda_0(\CC^*)$ is free and the map $p$ is $\lambda_0(\CC^*)$-invariant.

\begin{lemma}
The map $p \colon \A^{n+1} \smallsetminus V(z_1, \ldots , z_l) \to 
\Bl_{L_\sigma} \A^n$ is a $\lambda_0(\CC^*)$-torsor.
\end{lemma}
\begin{proof}
Observe that  $\Bl_{L_\sigma} \A^n = \Bl_0 \A^{l} \times \A^{n-l}$ and that the map
$p$ is obtained by base change from the map $q\colon \A^{l+1} \smallsetminus V(z_1, \ldots , z_l) \to 
\Bl_0 \A^{l}$, given by $(z_0, z_1, \ldots , z_l) \mapsto (z_0z_1, \ldots , z_0z_l) \times 
[z_1 \colon \ldots \colon z_l]$. To prove the lemma it suffices to show that $q$ is a $\CC^*$-torsor.
Again the action is free so it suffices to show that $\Bl_0 \A^l$ is the quotient.
To see this cover $\CC^{l+1} \smallsetminus V(z_1, \ldots , z_l)$ by the invariant open
affines $\Spec \CC[z_0, z_1, \ldots , z_l, 1/z_i]$ where $1 \leq i \leq l$. The subring
of invariants is $\Spec \CC[z_0z_1, \dots , z_0z_l, z_1/z_i, \ldots , z_l/z_i]$
which is the natural affine covering of $\Bl_0 \A^l$.
\end{proof}

Recall (Definition \ref{def.reich-l}) that we defined $Y'_\sigma$ to
be the Reichstein transform of $X(\widetilde{\Sigma})$ relative to the
linear subspace $L_\sigma$. We show below in Lemma \ref{lem.torsor} that $p^{-1}(Y'_\sigma) = X(\widetilde{\Sigma_\sigma})$. Theorem \ref{thm.technical} then follows: by Lemma \ref{lem.torsor} and exact sequence \eqref{seq.basicG} the
quotient stack ${\mathcal X}(\mbf{\Sigma_\sigma}):=
[X(\widetilde{\Sigma_\s})/G_{\mbf{\Sigma_\sigma}}]$
is equivalent to the quotient stack $[Y'_{\sigma}/G_{\mbf{\Sigma}}]$.

\begin{lemma} \label{lem.torsor}
$p^{-1}(Y'_\sigma) = X(\widetilde{\Sigma_\sigma})$. Hence
the map $p\vert_{X(\widetilde{\Sigma_\sigma})} \colon X(\widetilde
{\Sigma_\sigma}) \to Y'_\sigma$ is $\lambda_0(\CC^*)$-torsor.
\end{lemma}
\begin{proof}
First we prove $p^{-1}(Y'_\sigma) \subset X(\widetilde{\Sigma_\sigma})$.
Let ${\mathbf y} = (x_1, \ldots , x_l, x_{l+1}, \ldots , x_N) \times [y_1: \ldots : y_l]$ be a point of $Y'_\sigma$. 
We distinguish two cases.

Case I: ${\bf y}$ is not in the exceptional divisor. In this case 
at least one of the $x_i$ is non-zero and
$p^{-1}({\bf y}) = \{(t,t^{-1}x_1, \ldots t^{-1}x_l,x_{l+1}, \ldots, x_N)|t\neq 0\}$. 
Let ${\bf x} := (x_1, \ldots , x_l, x_{l+1}, \ldots , x_N)$.
Since ${\bf x} \in \A^n \smallsetminus (Z_\Sigma \cup \tilde{L}_\sigma)$
there is a cone
$\mu \subset \Sigma$ such that the function $x^{\hat{\mu}}$ does not vanish at
${\mathbf x}$. If $\mu$ does not contain $\sigma$, then the argument used at the beginning of the proof of Proposition \ref{prop.ltilde} implies that ${\bf x} \notin 
\tilde{L}_\sigma$ as well. Since $\mu$ does not contain $\sigma$ it is also a cone of the fan $\Sigma_\sigma$. Moreover, $t \neq 0$, so the function
$z^{\hat{\mu}}$ is non-zero on $p^{-1}({\bf y})$. Hence
$p^{-1}({\bf y}) \subset X(\widetilde{\Sigma_\sigma})$.

On the other hand, if $\mu$ contains $\sigma$ then the assumption
that ${\bf x} \notin \tilde{L}_\sigma$ implies that there is a face
$\tau \subset \mu$ such that $x^{\hat{\tau}}_\mu$ does not vanish at ${\bf x}$.
Thus there is a cone
$\mu \supset \sigma$ and a face $\tau \subset \mu$ not containing $\mu$
such that the function $x^{\hat{\tau}}_\sigma$ is non-zero at ${\bf x}$. 
Since $x^{\hat{\tau}}=
x^{\hat{\tau}}_\sigma  x^{\hat{\sigma}}$ it follows that
the function $x^{\hat{\tau}}$ also does not vanish at $x$. Let
$\tau' = \Cone(\tau,\rho_0)$ be the corresponding cone of the
fan $\Sigma_\sigma$. Then the function $z^{\hat{\tau'}}$ is non-zero
on $p^{-1}({\mathbf y})$.

Case II: ${\bf y}$  is on the exceptional divisor; i.e. $x_1 = \ldots = x_l=0$. In this case $p^{-1}({\bf y}) = \{(0, y_1^*,\ldots y^*_{l}, x_{l+1}, \ldots , x_N)|[y_1^*: \ldots y_l^*] = [y_1 :\ldots : y_l] \in
\Pro^{l-1}\}$.
Again ${\bf x} = (0,0, \ldots 0, x_{l+1}, \ldots x_n)$ is in
$X(\widetilde{\Sigma})$ so there is a cone $\mu$, 
necessarily
containing $\sigma$ such that $x^{\hat{\mu}}$ is non-zero at ${\mathbf x}$.
Also since ${\bf y}$ is not in the proper transform of $\tilde{L}_\sigma$
we know that ${\bf x^*} = (y_1^*, \ldots y_l^*,x_{l+1}, \ldots , x_n)$ is 
not in
$\tilde{L}_\sigma$. Thus there is a face $\tau$ of $\mu$ such 
the function $x^{\hat{\tau}}_\mu$ does not vanish at ${\bf x}^*$. Since
$x^{\hat{\mu}}$ also does not vanish at ${\bf x}^*$ again we conclude
that $x^{\hat{\tau}}$ is non-zero at ${\bf x}^*$.  If
we let $\tau'= \Cone(\tau, \rho_0)$ then again the function
$z^{\hat{\tau'}}$ is non-zero on $p^{-1}({\bf y})$. 

Therefore $p^{-1}(Y'_\sigma) \subset X(\widetilde{\Sigma_\sigma})$.

The proof that $p(X(\widetilde{\Sigma_\sigma})) \subset
Y'_\sigma$ is similar. 
Let ${\mathbf z} = (z_0, z_1, \ldots , z_l, z_{l+1}, \ldots z_n)$ be point of
$X(\widetilde{\Sigma_\sigma})$.
Again we have two cases.

Case I: $z_0 \neq 0$. Thus $p({\mathbf z})$ is not contained in the exceptional divisor.
In this case it suffices to show that ${\mathbf x} = 
(z_0z_1, \ldots , z_0z_l, z_{l+1}, \ldots z_{n})$ 
is in the complement of $Z_\Sigma \cup \tilde{L}_\sigma$. 
By hypothesis there is a cone $\tau'$ of $\Sigma_\sigma$ such
that $z^{\hat{\tau'}} \neq 0$ at ${\mathbf z}$. If $\tau'$ corresponds to a cone
$\tau$ of $\Sigma$ not containing $\sigma$ then ${\mathbf x} \notin \tilde{L}_\sigma$ and the function $x^{\hat{\tau}}$ doesn't vanish at ${\mathbf x}$, so
$p({\mathbf z}) \in Y'_\sigma$.

On the other hand suppose that $\tau' = \Cone(\tau, \sigma)$. Then
the function $x^{\hat{\tau}} \neq 0$ on $p({\mathbf z})$.
For every cone $\mu$ which contains $\tau$ and $\sigma$ 
we may factor $x^{\hat{\tau}} = x^{\hat{\mu}}x^{\hat{\tau}}_\mu$.
Thus for every $\mu \supset \sigma$, $x^{\hat{\tau}}_\mu$ doesn't
vanish at ${\mathbf x}$ so $p({\mathbf z}) \notin \tilde{L}_\sigma$ as well.

Case II: $z_0 = 0$. Here $p({\mathbf z}) = (0,0,\ldots 0, z_{l+1},
\ldots , z_n) \times [z_1: \ldots : z_l]$ is contained in the exceptional divisor. By hypothesis there is a cone $\tau'$ of $\Sigma_\sigma$ such that
$z^{\hat{\tau'}} \neq 0$ at ${\bf z}$. Since $z_0 =0$ it follows that
$\rho_0$ must be a ray of $\tau'$. Hence $\tau' = \Cone(\tau, \rho_0)$.
Thus for every cone $\mu$ containing $\tau$ the functions
$x^{\hat{\mu}}$ and $x^{\hat{\tau}}_\mu$ do not vanish
at ${\mathbf x} =(z_1, \ldots , z_l, z_{l+1}, \ldots z_n) \in \A^n$.
Hence $p({\mathbf z}) = (0,0, \ldots , 0, z_{l+1}, \ldots , z_n)$ is not in the proper transform of $\tilde{L}_\sigma$. Also for
any cone $\mu$ containing $\tau$ and $\sigma$ the function
$x^{\hat{\mu}}$ doesn't vanish at $(0,0,\ldots 0, z_{l+1}, \ldots z_n)$.
Therefore, $p({\mathbf z}) \in Y'_\sigma$.

This completes the proof of Lemma \ref{lem.torsor}, and thus Theorem \ref{thm.technical}.
\end{proof}

\section{Relation with Kirwan-Reichstein partial desingularization}
 
Kirwan \cite{Kir:85} gave a method for obtaining partial
desingularizations of geometric invariant theory (GIT) quotients of
smooth projective varieties.
Let $X$ be a smooth projective variety with the linearized action of a reductive group $G$. Let $L$ be a $G$-linearized ample line bundle and let 
$X^{ss} = X^{ss}(L)$ and $X^{s} = X^s(L)$ be the sets of semi-stable and stable points.

Kirwan's procedure for partial desingularization can be described as a canonical sequence of Reichstein
transformations.  Framed in this language Kirwan's main result is the
following.
\begin{theorem} \cite{Kir:85}, \cite[Chapter 8, Section 4]{MFK:94}
  With the notation as above let $r$ be the maximum dimension of a
  reductive subgroup of $G$ that fixes a point of $X^{ss}$.

(i) If $r = 0$ then every semi-stable point stable.

Let
  $Z_r \subset X^{ss}$ be the locus of points fixed by a reductive
  subgroup of dimension $r$.

(ii) $Z_r \subset X^{ss}$ is non-singular.

Let $X_{r-1} = {\mathcal R}(X^{ss}, Z_r)$. Then the following hold:

(iii) Any reductive subgroup of $G$ that fixes a point of $X_{r-1}$ has
dimension strictly less than $r$.

(iv) The morphism ${\mathcal R}(X^{ss}, Z_r) \to X^{ss}$ is an isomorphism
over the open set $X^s\subset X^{ss}$.

(v) There exists a $G$-equivariant projective birational morphism $X' \to X$ with $X'$ smooth and such that $X_{r-1} = (X')^{ss}(L')$ for a suitable line
bundle $L'$ on $X'$. In particular there is a projective 
good quotient $X_1/G$ and projective birational morphism $X_1/G \to
X^{ss}/G$ which is an isomorphism over the open set $X^{s}/G \subset X^{ss}/G$.

Hence after iterating this process a finite number of times every
semi-stable point becomes stable and we get a $G$-equivariant
birational morphism $X_0 \to X$ and a smooth DM stack $\mc{X}'=
[X_0/G]$. The projective variety $X_0/G$ has finite quotient
singularities and the map $X_0/G \to X^{ss}/G$ is projective and
birational and an isomorphism over the quotient of the stable locus
$X^s/G \subset X^{ss}/G$.

\end{theorem}

We now come to the main result of this paper - 
an analogue of Kirwan's result for
toric stacks. Since toric varieties need not be projective or quasi-projective, toric
stacks are not necessarily of the form $[X^{ss}/G]$ and so in these cases Kirwan's result does not
apply. In a subsequent paper we will consider analogues of Kirwan's
result for arbitrary quotient stacks which admit good moduli spaces.

\begin{theorem}[Partial desingularization of Artin toric stacks] \label{thm.torickirwan}
Let $\mathbf{\Sigma} = (N,\Sigma, \beta)$ be a stacky fan.  Then there is a simplicial stacky fan ${\mathbf \Sigma'} = (N,\Sigma', \beta')$ and a birational map of toric stacks
${\mathcal X}({\mathbf \Sigma'}) \to {\mathcal X}({\mathbf \Sigma)}$
which induces a proper birational morphism of toric varieties 
$X(\Sigma') \to X(\Sigma)$.

The map ${\mathcal X}({\mathbf \Sigma'}) \to {\mathcal X}({\mathbf
  \Sigma})$ is obtained by a finite sequence of stacky star subdivisions
along minimal non-simplicial cones. In addition, ${\mathcal X}({\mathbf
  \Sigma'})$ is isomorphic to the stack obtained by a finite sequence
of Reichstein transformations of the locus of points maximal
dimensional stabilizer.

If $X(\Sigma)$ is a projective toric variety, then this sequence of
Reichstein transformations is the same as the one in Kirwan's theorem applied
to a weighted projective space.
\end{theorem}

\begin{proof}

We proceed by induction on the non-simplicial index $\ns(\Sigma)$ (Definition \ref{def.nonsimplicial}).   
If $\ns(\Sigma)=0$, then $\Sigma$ is simplicial, and there is nothing to prove. 

Assume $\ns(\Sigma)>0$. Let $S=S_\Sigma$ be the set of minimal
non-simplicial cones of $\Sigma$ (Definition
\ref{def.nonsimplicial}). For $\s \in \Sigma$, recall (Definitions
\ref{def.l}, \ref{def.tilde-ls}) that $L_\s=\mathbb{V}(x_\rho \mid
\rho \in \s(1)) \cap X(\widetilde{\Sigma})$ and
$\tilde{L}_{\s}=q^{-1}({q(L_\s)})=q^{-1}(V(\s))$ where
$q:{X}(\widetilde{\Sigma}) \to X(\Sigma)$ is the quotient map.

Note that for any two distinct elements $\s, \s' \in S$, we have
$V(\s) \cap V(\s')= \emptyset$: if $p \in V(\s) \cap V(\s')$ then by
the orbit-cone correspondence (Section \ref{section.toric}, equation
\ref{eq.orbit-cone-corr}) $p$ is in the orbit $O(\tau)$ for a cone
$\tau$ containing $\s$ and $\s'$, but by Lemma \ref{lemma.ns3}, no
such cone $\tau$ exists.

It follows that for any two distinct elements $\s, \s' \in S$,
$\tilde{L}_{\s}$ and $\tilde{L}_{\s'}$ are disjoint (since if $p$ were
a common point, then $q(p) \in V(\s) \cap V(\s')$). Let
${L}_S=\cup_{\s \in S}{L}_\s$, and let $\tilde{L}_S=\cup_{\s \in S}
\tilde{L}_\s$; both of these unions are in fact disjoint unions.  The following calculation shows
that $\tilde{L}_S=q^{-1}({q(L_S)})$: $${q(L_S)} = q(\coprod_{\s \in
    S}{L}_\s)=\coprod_{\s \in S}{q({L}_\s)}=\coprod_{\s \in S}
V(\s)$$
$$q^{-1}({q(L_S)})=q^{-1}(\coprod_{\s \in S} V(\s))=\coprod_{\s \in S} q^{-1}(V(\s))=\coprod_{\s \in S} \tilde{L}_\s=\tilde{L}_S.$$

Let $Y_S$ be the blowup of $\mc{X}(\Sigma)$ along ${L}_S$, and let
$Y_S'$ be the complement in $Y_S$ of the strict transform of
$\tilde{L}_S$. In other words, $Y_S'$ is the Reichstein transform of
$L_S \subseteq \mc{X}(\Sigma)$. By Lemma \ref{stabilizer-size}, the
general point of $L_S$ has $G_{\mbf{\Sigma}}$-stabilizer of rank
$\ns(\Sigma)$, and by the same lemma, $L_S$ is the closure of the locus
of maximal dimensional stabilizer.

Since the $\tilde{L}_\s$ are disjoint as $\s \in S$ varies, $Y_S'$ is
isomorphic to successively blowing up $\mc{X}(\Sigma)$ along each
$L_\s$ and removing the strict transform of $\tilde{L}_\s$, as $\s \in
S$ varies. Let $\mbf{\Sigma_S}$ be the stacky fan formed by starting
with $\mbf{\Sigma}$ and stacky star subdividing (Definition
\ref{def.star}) each cone $\s \in S$. By repeatedly applying
Theorem \ref{thm.technical}, we conclude that
$[Y_S'/G_{\mbf{\Sigma}}]$ is isomorphic to the toric stack
$\mc{X}(\mbf{\Sigma_S})$. By Lemma \ref{ns-drops}, we have
$\ns(\Sigma_S) < \ns(\Sigma)$, and the inductive step is complete.

If $X(\Sigma)$ is projective, then by \cite[Prop. 14.1.9]{CLS:10} there is a $G=G_{\mbf{\Sigma}}$-equivariant line bundle $L_{{\mathbf a}}$
on $\A^n$ (where $n=\vert \Sigma(1)\vert$) such that $X(\tilde{\Sigma}) = \A^n(L_{{\mathbf
    a}})^{ss}$. Hence the argument of \cite[Proposition 12.2]{Dol:03}
expresses $X(\Sigma)$ as a GIT quotient of a weighted projective
space.  Hence we can apply Kirwan's result. Kirwan's algorithm
proceeds by successively taking the Reichstein transforms of the locus
of maximal dimensional stabilizer, and hence our algorithm agrees with
Kirwan's in the projective toric case.
\end{proof}

\section{Divisorial Reichstein transformations and change of
  linearizations} \label{sec.divreich} Recall \cite{Tha:96, DoHu:98}
that if $X$ is projective variety with the action of a reductive group
$G$ then the cone of $G$-linearized ample divisors is divided into a
finite number of chambers such that the GIT quotient is constant on
the interior of each chamber. In many (but not all) examples if $L$ is
a line bundle corresponding to a point on a wall between two chambers
then $X^{ss}(L) \neq X^{s}(L)$, but if $L$ is deformed to a line
bundle $L'$ in the interior of a chamber then $X^{s}(L') = X^{ss}(L')
\subset X^{ss}(L)$.  From the stack point of view this means that the
non-separated quotient stack $[X^{ss}(L)/G]$ with complete good moduli
space $X^{ss}(L)/G$ contains the complete Deligne-Mumford (DM) open
substack $[X^{s}(L')/G]$. The induced map of quotients $X^{s}(L')/G
\to X^{ss}(L)/G$ is proper and birational.  Our final result shows we
can use Reichstein transformations relative to divisors to find
complete open DM substacks of Artin toric stacks with complete good
moduli space. In our subsequent paper we will further investigate the
relationship between divisorial Reichstein transformations and changes
of linearizations.

\begin{theorem} \label{thm.divsub}
Let ${\mathbf \Sigma} = (N, \Sigma, (v_1, \ldots , v_n))$ be a stacky
fan  and let ${\mathcal X}({\mathbf \Sigma})$ be the associated toric
stack.
Then there exists a sequence of divisorial toric Reichstein transforms relative to rays in $\Sigma$ such that the resulting
toric stack ${\mathcal X}({\mathbf \Sigma'})$ is Deligne-Mumford.
In particular if the toric variety $X(\Sigma)$ is complete,
then the toric stack ${\mathcal X}({\mathbf \Sigma})$ contains a
complete open Deligne-Mumford substack ${\mathcal X}({\mathbf \Sigma'})$ such that the induced map on
moduli spaces $X(\Sigma') \to X(\Sigma)$ is proper and birational.
\end{theorem}

\begin{proof}
Our result follows from \cite[Proposition 11.1.7]{CLS:10}
which says that any fan can be made simplicial by a sequence of star
subdivisions relative to the rays of the fan. Let $\Sigma'$ be a fan
obtained from $\Sigma$ by subdividing relative to a sequence of 
rays $\rho_1, \ldots , \rho_l$. If ${\mathbf \Sigma}' = (N, \Sigma',
(v_1, \ldots , v_n))$ then the toric stack ${\mathcal X}({\mathbf
  \Sigma}')$
is obtained from ${\mathcal X}({\mathbf \Sigma})$ by a sequence of
divisorial Reichstein transformations.
\end{proof}

\begin{remark}
  If $X(\Sigma)$ is projective then $X(\widetilde{\Sigma}) =
  (\A^n)^{ss}(\chi)$ for some character $\chi \in \widehat{G_{\mathbf
      \Sigma}}$. If ${\mathcal X}({\mathbf \Sigma}') \subset {\mathcal
    X}({\mathbf \Sigma})$ is an open toric substack obtained by a
  divisorial Reichstein transform, then $X(\widetilde{\Sigma'}) =
  (\A^n)^{ss}(\chi')$ for some other character $\chi' \in
  \widehat{G_{\mathbf \Sigma}}$. Thus the toric variety $X(\Sigma')$
  is obtained by a change of linearization for the $G_{\mathbf \Sigma}
  = G_{\mathbf \Sigma'}$ action on $\A^n$.
\end{remark}

If $G \subset (\CC^*)^n$ and $\chi \in \widehat{G}$ is a character then
the quotient $(\A^n)^{ss}(\chi)/G$ is again a toric variety although
the stack $[(\A^n)^{ss}/G]$ need not be a toric stack in our sense of
the term. As $\chi$ varies through $\widehat{G}$ the different quotients
are related by birational transformations. As the next example shows, these transformations can also be interpreted
in terms of divisorial Reichstein transforms.
\begin{example}
The group
$G=\{(t, t^{-1}u, t, u) \vert t,u \in \CCe\} \subseteq (\CCe)^4$
acts on $\A^4=\Spec \CC[x_1, x_2, x_3, x_4]$ in the natural way.
Since $G \simeq (\CC^*)^2$ the character group $\widehat{G}$ is the lattice $\ZZ^2$. As noted in  \cite[Example 14.3.7]{CLS:10}
the GIT quotient $\A^4//_{\chi}G$ is non-empty if and only $\chi$
lies in the cone spanned by $(-1,1)$ and $(1,0)$.
This cone is divided into two chambers separated by the wall spanned by the ray through $(0,1)$.
In the interior of the left chamber, as well as on the
wall, the quotients are all isomorphic to $\Pro^2$, and in the interior of the right
chamber the quotients are all isomorphic to Hirzebruch surface  $H_1$.
 
The point $\beta=(0,1)$ lies on the wall and the GIT
quotient $\A^4 //_{(0,1)} G$, is equal to $\Proj \CC[x_1x_2, x_2x_3,
x_4] \simeq \Pro^2$. This is the good quotient of 
$$X=\A^4 \smallsetminus \left(\VV(x_1, x_3, x_4)
\bigcup \VV(x_2, x_4)\right)$$ by $G$.  
Note however, that the
quotient stack $[X/G]$ is not a toric stack in the sense of this
paper, although it is a toric stack in the sense of \cite{Sat:09}.
The reason is that $[X/G]$ is not a DM stack (since $(0,0,0,1)
\in X$ has a 1-dimensional stabilizer) but the moduli space of $[X/G]$
is a simplicial toric variety.

Consider the point $(-1,2)$ in the interior of the left
chamber.  The semi-stable locus relative to this character is $\A^4
\smallsetminus \left(\VV(x_1,x_3,x_4) \bigcup \VV(x_2)\right)$ and the GIT
quotient $\A^4//_{(-1,2)}G$ is $\Pro^2$.  The Reichstein transform of
$X$ relative to the divisor $x_1=0$ is 
$\A^4
\smallsetminus \left(\VV(x_1,x_3,x_4) \bigcup \VV(x_2)\right)$
since the saturation (relative
to the quotient map $q \colon X \to X/G$) of the divisor $x_1=0$ is
$\VV(x_1)\bigcup \VV(x_2)$. 
Thus the Reichstein transform of the
non-toric stack $[X/G]$ relative to divisor $[\VV(x_0)/G]$ is the
representable toric stack $[\left(\A^4 \smallsetminus
  (\VV(x_1,x_3,x_4) \bigcup \VV(x_2))\right)/G]$ which is
represented by the smooth toric variety $\Pro^2$.

Now take the point $(1,1)$ in the right chamber. The semi-stable locus
relative to this character is $\A^4\smallsetminus (\VV(x_1, x_3)
\bigcup \VV(x_2, x_4))$. The GIT quotient $\A^4//_{(1,1)}G$ is
$H_1$.  The Reichstein transform of $X$ relative to the divisor
$x_2=0$ is 
$$\A^4\smallsetminus \left( \VV(x_1, x_3)\bigcup \VV(x_2,x_4)\right)$$ 
since the saturation of the
divisor $x_2=0$ is $\VV(x_1, x_3)\bigcup \VV(x_2)$. Thus we see again that a
divisorial Reichstein transform of the non-toric stack $[X/G]$
produces the new GIT quotient. The quotient stack
$[\A^4((1,1))^{ss}/G]$ is again a representable toric stack which is
represented by the toric variety $H_1$.

This example will be generalized in a subsequent paper.
\end{example}

Our final example gives a complete open
DM substack $\mathcal{X}(\Sigma'')$ of a Cox stack
$\mathcal{X}(\Sigma)$ with projective moduli space $X(\Sigma)$ that
cannot be formed by a sequence of Reichstein transforms starting with
$\mathcal{X}(\Sigma)$. However $\mathcal{X}(\Sigma'')$ is the blowdown
(of a Reichstein transform) of a stack that is the result of a
sequence of Reichstein transforms starting with $\mathcal{X}(\Sigma)$.

This example shows that there are projective or quasi-projective varieties $X^{ss}(L)/G$ such
that $X^{ss}(L)$ contains open sets which have good geometric quotients
but which are not GIT quotients. These quotients do not fit into the
chamber decomposition description of \cite{Tha:96, DoHu:98}.

\begin{example} \label{ex.secretchamber} 
We begin by copying verbatim
  the setup of \cite[Example 6.1.17]{CLS:10}. The fan for $\Pro^1
  \times \Pro^1 \times\Pro^1$ has the eight orthants of $\mathbb{R}^3$
  as its maximal cones, and the ray generators are $\pm e_1,\pm e_2,
  \pm e_3$. Take the positive orthant $\mathbb{R}^3_{\geq 0}$ and
  subdivide it further by adding new ray generators $$a=(2,1,1), \
  b=(1,2,1), \ c=(1,1,2), \ d=(1,1,1).$$ We obtain a complete fan
  $\Sigma$ by filling the first orthant with the maximal
  cones $$\sigma_1=\langle e_1,e_2, a, b\rangle, \ \sigma_2=\langle
  e_2, e_3, b, c\rangle, \ \sigma_3=\langle e_1, e_3, a, c\rangle, \
  \langle a, b, d \rangle,\ \langle b,c,d \rangle, \ \langle a, c, d
  \rangle .$$ The complete toric variety $X(\Sigma)$ is a projective
  variety (a calculation shows that the divisor $D=16(D_1 + D_2 + D_3)
  + 56(D_a + D_b+D_c) + 41D_d$ on $X(\Sigma)$ is ample.)  In the
  positive orthant, the intersection of the plane $x+y+z=1$ with
  $\Sigma$ looks like: $
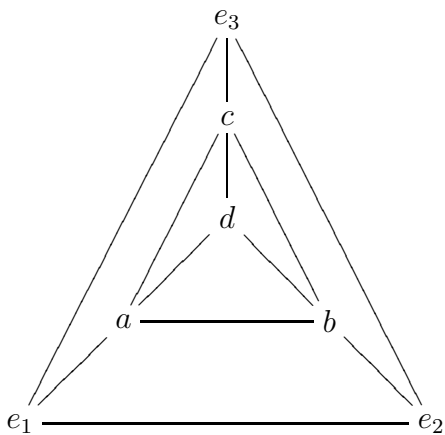
\xymatrix{
    &&e_3 \ar@{-}[ddddll]\ar@{-}[d]\ar@{-}[ddddrr]\\
    &&c\ar@{-}[ddl]\ar@{-}[d]\ar@{-}[ddr] \\
    &&d \ar@{-}[dl]\ar@{-}[dr]\\
    &a \ar@{-}[rr]\ar@{-}[dl] &&b \ar@{-}[dr]\\
    e_1 \ar@{-}[rrrr] &&&&e_2\\
  }$

  If we perform a star subdivision of $\Sigma$ with respect to the ray
  through $a$, followed by a star subdivision by the ray through $b$,
  we get a fan $\Sigma'$ formed from by $\Sigma$ by dividing each of
  the maximal cones $\sigma_1, \sigma_2, \sigma_3 \in \Sigma$ into two
  maximal cones separated by the facets $\langle a, e_2\rangle,\
  \langle b, e_3\rangle,\ \langle a, e_3\rangle$ respectively. The
  associated Cox stack $\mathcal{X}(\Sigma')$ is a complete open DM
  substack of the Cox stack $\mathcal{X}(\Sigma')$, and
  $\mathcal{X}(\Sigma')$ is isomorphic to a sequence of divisorial
  Reichstein transforms starting with
  $\mathcal{X}(\Sigma)$. Furthermore, we have a projective birational
  morphism $X(\Sigma') \to X(\Sigma)$, since a star subdivision of a
  fan gives rise to a projective morphism on the level of toric
  varieties (\cite[Theorem 11.1.6]{CLS:10}).

In the positive orthant, the intersection of the plane $x+y+z=1$ with $\Sigma'$ looks like:
$
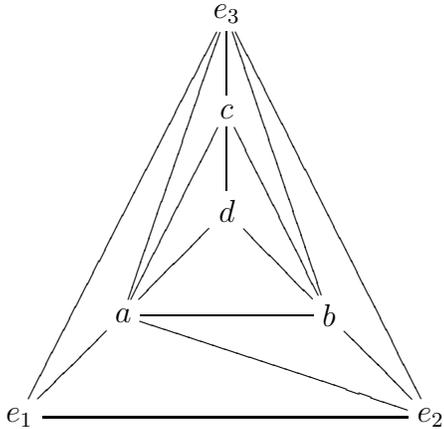
\xymatrix{
&&e_3 \ar@{-}[ddddll]\ar@{-}[d]\ar@{-}[ddddrr]\ar@{-}[dddl]\ar@{-}[dddr]\\
&&c\ar@{-}[ddl]\ar@{-}[d]\ar@{-}[ddr] \\
&&d \ar@{-}[dl]\ar@{-}[dr]\\
&a \ar@{-}[rr]\ar@{-}[drrr]\ar@{-}[dl] &&b \ar@{-}[dr]\\
e_1 \ar@{-}[rrrr] &&&&e_2\\
}$

However, not all complete open DM substacks of $\mathcal{X}(\Sigma)$
can be reached by Reichstein transforms of $\mathcal{X}(\Sigma)$. Form
the fan $\Sigma''$ from $\Sigma'$ flipping the facet subdividing
$\sigma_3$ from $\langle a, e_3\rangle$ to $\langle c, e_1\rangle$. 

In the positive
orthant, the intersection of the plane $x+y+z=1$ with $\Sigma''$ looks
like: $
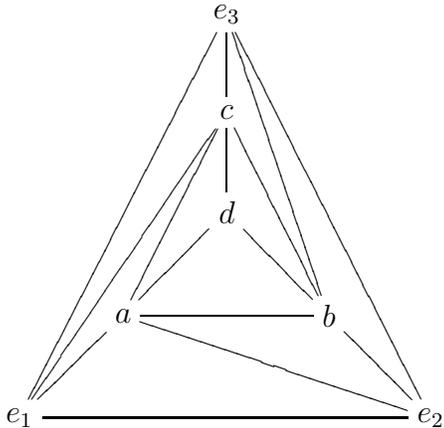
\xymatrix{
  &&e_3 \ar@{-}[ddddll]\ar@{-}[d]\ar@{-}[ddddrr]\ar@{-}[dddr]\\
  &&c\ar@{-}[ddl]\ar@{-}[d]\ar@{-}[ddr]\ar@{-}[dddll] \\
  &&d \ar@{-}[dl]\ar@{-}[dr]\\
  &a \ar@{-}[rr]\ar@{-}[drrr]\ar@{-}[dl] &&b \ar@{-}[dr]\\
  e_1 \ar@{-}[rrrr] &&&&e_2\\
}$

Since the toric variety $X(\Sigma'')$ is not projective
\cite[Example 6.1.17]{CLS:10}  $X(\Sigma'')$ cannot
be obtained from $X(\Sigma)$ by star subdivisions.
Hence the Cox stack $\mathcal{X}(\Sigma'')$ is a
complete open Deligne-Mumford substack of $\mathcal{X}(\Sigma)$ which 
cannot be reached by Reichstein transforms of $\mathcal{X}(\Sigma)$.  
\end{example}


\begin{thebibliography}{MFK}

\bibitem[Alp]{Alp:08}
Jarod Alper, {\em Good moduli spaces for {A}rtin stacks}, math.AG/08042242
  (2008).

\bibitem[BCS]{BCS:05}
Lev~A. Borisov, Linda Chen, and Gregory~G. Smith, {\em The orbifold {C}how ring
  of toric {D}eligne-{M}umford stacks}, J. Amer. Math. Soc. \textbf{18} (2005),
  no.~1, 193--215 (electronic).

\bibitem[CLS]{CLS:10}
David Cox, John Little, and Hal Schenk, {\em Toric varieties}, American
  Mathematical Society, Providence, RI, 2011, Graduate studies in mathematics.

\bibitem[Cox]{Cox:95}
David~A. Cox, {\em The homogeneous coordinate ring of a toric variety}, J.
  Algebraic Geom. \textbf{4} (1995), no.~1, 17--50.

\bibitem[Dol]{Dol:03}
Igor Dolgachev, {\em Lectures on invariant theory}, London Mathematical Society
  Lecture Note Series, vol. 296, Cambridge University Press, Cambridge, 2003.

\bibitem[DH]{DoHu:98}
Igor~V. Dolgachev and Yi~Hu, {\em Variation of geometric invariant theory
  quotients}, Inst. Hautes \'Etudes Sci. Publ. Math. (1998), no.~87, 5--56,
  With an appendix by Nicolas Ressayre.

\bibitem[FMN]{FMN:07}
Barbara Fantechi, Etienne Mann, and Fabio Nironi, {\em Smooth toric
  {D}eligne-{M}umford stacks}, J. Reine Angew. Math. \textbf{648} (2010),
  201--244.


\bibitem[Ful]{Ful:93}
William Fulton, {\em Introduction to toric varieties}, Annals of Mathematics
  Studies, vol. 131, Princeton University Press, Princeton, NJ, 1993, The
  William H. Roever Lectures in Geometry.

\bibitem[Iwa]{Iwa:09}
Isamu Iwanari, {\em The category of toric stacks}, Compos. Math. \textbf{145}
  (2009), no.~3, 718--746.

\bibitem[KM]{KeMo:97}
Se{\'a}n Keel and Shigefumi Mori, {\em Quotients by groupoids}, Ann. of Math.
  (2) \textbf{145} (1997), no.~1, 193--213.

\bibitem[Kir]{Kir:85}
Frances~Clare Kirwan, {\em Partial desingularisations of quotients of
  nonsingular varieties and their {B}etti numbers}, Ann. of Math. (2)
  \textbf{122} (1985), no.~1, 41--85.

\bibitem[Kol]{Kol:97}
J{\'a}nos Koll{\'a}r, {\em Quotient spaces modulo algebraic groups}, Ann. of
  Math. (2) \textbf{145} (1997), no.~1, 33--79.

\bibitem[MFK]{MFK:94}
D.~Mumford, J.~Fogarty, and F.~Kirwan, {\em Geometric invariant theory}, third
  ed., Springer-Verlag, Berlin, 1994.

\bibitem[Rei]{Rei:89}
Zinovy Reichstein, {\em Stability and equivariant maps}, Invent. Math.
  \textbf{96} (1989), no.~2, 349--383.

\bibitem[Sat]{Sat:09}
Matthew Satriano, {\em Canonical {A}rtin stacks over log smooth schemes},
  arXiv:0911.2059.

\bibitem[Ses]{Ses:72}
C.~S. Seshadri, {\em Quotient spaces modulo reductive algebraic groups}, Ann.
  of Math. (2) \textbf{95} (1972), 511--556; errata, ibid. (2) 96 (1972), 599.

\bibitem[Tha]{Tha:96}
Michael Thaddeus, {\em Geometric invariant theory and flips}, J. Amer. Math.
  Soc. \textbf{9} (1996), no.~3, 691--723.

\bibitem[W{\l}o]{Wlo:97}
Jaros{\l}aw W{\l}odarczyk, {\em Decomposition of birational toric maps in
  blow-ups \& blow-downs}, Trans. Amer. Math. Soc. \textbf{349} (1997), no.~1,
  373--411.

\end{thebibliography}
\def\cprime{$'$}
\def\cprime{$'$}

\end{document}